\RequirePackage{fix-cm}

\documentclass[10pt,a4,envcountsect]{svjour3}                     % onecolumn (standard format)
\smartqed  % flush right qed marks, e.g. at end of proof
\usepackage{graphicx}
\usepackage{amsmath}
\usepackage{amssymb}
\usepackage{mathptmx}
\usepackage{multicol}
\usepackage{ifpdf}
\usepackage{hyperref}
\usepackage{url}
\newcounter{algorithmenumi}
\newenvironment{algorithm}[1]
{\vskip0.3cm\noindent\textbf{Algorithm }\refstepcounter{algorithmenumi}\textbf{\arabic{section}}.\textbf{\arabic{algorithmenumi}}. $($#1$)$.}{\textbf{End.}}
\usecounter{algorithmenumi}
\newtheorem{assumption}{Assumption A.\!}[section]
\newcommand{\aref}[1]{\textrm{\textsc{A.\ref{#1}}}}
\newcommand{\norm}[1]{\|#1\|}
\newcommand{\abs}[1]{|#1|}
\newcommand{\set}[1]{\left\{#1\right\}}
\newcommand{\Eproof}{\hfill$\square$}

\journalname{J. Optim. Theory Appl.}

\begin{document}

\title{Fast Inexact Decomposition Algorithms for Large-Scale Separable Convex Optimization}
% \title{Insert your title here%\thanks{Grants or other notes
% \subtitle{Do you have a subtitle?\\ If so, write it here}

\titlerunning{Fast Inexact Decomposition Algorithms For Large-Scale Separable Convex Optimization}

\author{Quoc Tran-Dinh \and Ion Necoara \and Moritz Diehl}


\institute{Quoc Tran Dinh \and Moritz Diehl are with Department of Electrical Engineering (ESAT-SCD) and Optimization in Engineering Center (OPTEC),
KU Leuven, Kasteelpark Arenberg 10, B-3001 Heverlee, Belgium.\\
\email{\texttt{\{quoc.trandinh, moritz.diehl\}@esat.kuleuven.be}}\\
Ion Necoara {is with the Automation and Systems Engineering Department, University Politehnica Bucharest, 060042 Bucharest, Romania;
\email{\texttt{ion.necoara@acse.pub.ro}}}\\
Quoc Tran Dinh {is with VNU University of Science, Hanoi, Vietnam. Currently, he is working as a postdoctoral researcher at Laboratory for Information and
Inference Systems, EPFL, 1015-Lausanne, Switzerland.}
}

\date{Received: date / Accepted: date}

\maketitle

\begin{abstract}
In this paper we propose a new inexact dual decomposition algorithm for solving separable convex optimization problems. This algorithm is a combination of three
techniques: dual Lagrangian decomposition, smoothing and excessive gap.  The algorithm requires only one primal step and two dual steps at each iteration and
allows one to solve the subproblem of each component inexactly and in parallel.
Moreover, the algorithmic parameters are updated automatically without any tuning strategy as in augmented Lagrangian approaches. We analyze the convergence of
the algorithm and estimate its $O\left(\frac{1}{\varepsilon}\right)$ worst-case complexity.
Numerical examples are implemented to verify the theoretical results.

\end{abstract}

%////////////////////////////////////////////////////////////////////////////////////////////////////////////////////////////////////
%+1. Introduction
%////////////////////////////////////////////////////////////////////////////////////////////////////////////////////////////////////
\section{Introduction}
Many practical optimization problems must be addressed  within the framework of large-scale structured convex optimization and need to
be solved in a parallel and distributed manner. Such problems may appear in many fields of science and engineering: e.g.  graph
theory, networks, transportation, distributed model predictive control, distributed estimation and multistage stochastic optimization, see e.g.
\cite{Bertsekas1996b,Holmberg2001,Kojima1993,Komodakis2010,Purkayastha2008,Samar2007,Tsiaflakis2010,Venkat2008,Zhao2005} and the references quoted therein.
Solving large-scale optimization problems is still a challenge in many applications \cite{Connejo2006} due to the limitations of computational devices
and computer systems. Recently, thanks to the development of parallel and distributed computer systems, many large-scale problems have been solved by using the
framework of decomposition. However, methods and algorithms for solving this type of problems, which can be performed in a parallel or distributed manner, are
still limited \cite{Bertsekas1989,Connejo2006}.

In this paper we develop a new optimization algorithm to solve the following structured convex optimization problem with a separable objective function and
coupling linear constraints:
\begin{equation}\label{eq:primal_cvx_prob}
\phi^{*} := \left\{ \begin{array}{cl}
\displaystyle\min_{x\in\mathbb{R}^{n}} & \Big\{\phi(x) := \displaystyle\sum_{i=1}^M\phi_i(x_{i})\Big\}\\
\textrm{s.t.}~ &x_{i} \in X_{i}~ (i=1,\dots, M),\\
& \displaystyle\sum_{i=1}^M(A_{i}x_{i} - b_i) = 0,
\end{array}\right.
\end{equation}
where, for every $i\in\set{1,2,\dots, M}$, $\phi_i:\mathbb{R}^{n_i} \to \mathbb{R}$ is convex (not necessarily strictly convex) and possibly \textit{nonsmooth}
functions, $X_{i} \in\mathbb{R}^{n_i}$ is nonempty, closed and convex sets, $A_{i}\in\mathbb{R}^{m\times n_i}$ and $b_i\in\mathbb{R}^m$, and
$n_1+n_2 +\cdots + n_M = n$. Here, $x_i\in X_i$ is referred to as a local convex constraint and the final constraint is called \textit{coupling linear
constraint}.

In the literature, several approaches based on decomposition techniques have been proposed to solve problem \eqref{eq:primal_cvx_prob}. In order to observe
the differences between those methods and our approach in this paper, we briefly classify some of these that we found most related. The first class
of algorithms is based on Lagrangian relaxation and subgradient methods of multipliers \cite{Bertsekas1989,Duchi2012,Nedic2009}. It has been observed that
subgradient methods are usually slow and numerically sensitive to the choice of step sizes in practice \cite{Nesterov2004}. Moreover, the convergence rate of
these methods is in general $O(1/\sqrt{k})$, where $k$ is the iteration counter.
The second approach relies on augmented Lagrangian functions, see e.g. \cite{Hamdi2005,Ruszczynski1995}. Many variants were proposed and tried to process the
inseparability of the crossproduct terms in the augmented Lagrangian function in different ways. Besides this approach, the authors in \cite{Han1988} considered
the dual decomposition based on Fenchel's duality theory. Another research direction is based on alternating direction methods which were studied, for example,
in \cite{Boyd2011,Fukushima1992,Goldfarb2010,Kontogiorgis1996}. Alternatively, proximal point-type methods were extended to the decomposition framework, see,
e.g. \cite{Chen1994,Tseng1997}. Other researchers employed interior point methods in the framework of decomposition such as
\cite{Kojima1993,Mehrotra2009,Necoara2009,TranDinh2012c,Zhao2005}.
Furthermore, the mean value cross decomposition in \cite{Holmberg2006}, the partial inverse method in \cite{Spingarn1985} and the accelerated gradient method of
multipliers in \cite{Necoara2008} were also proposed to solve problem \eqref{eq:primal_cvx_prob}. We note that decomposition and splitting methods are very
well developed in convex optimization, especially in generalized equations and variational inequalities, see e.g.
\cite{Combettes2008,Eckstein1992,BricenoArias2011}.
Recently, we have proposed a new decomposition method to solve problem \eqref{eq:primal_cvx_prob} in  \cite{TranDinh2012a} based on two primal steps and one
dual step. It is proved that the convergence rate of the algorithm is $O(1/k)$ which is much better than the subgradient-type methods of multipliers
\cite{Bertsekas1989} but its computational complexity per iteration is higher that of these classical methods. Moreover, the algorithm uses an
automatic strategy to update the parameters which improves the numerical efficiency in practice.

In this paper, we propose a new inexact decomposition algorithm for
solving \eqref{eq:primal_cvx_prob} which employs  smoothing
techniques \cite{Nesterov2005} and excessive gap condition
\cite{Nesterov2005a}.

\vskip0.1cm
\noindent\textbf{Contribution. }
The contribution of the paper is as follows:
\begin{enumerate}
\item We propose a new  decomposition algorithm based on inexact dual gradients.  This algorithm requires only one primal step and two dual steps at each
iteration and allows one to solve the subproblem of each component inexactly and in parallel. Moreover, all the algorithmic parameters are updated
automatically without using any tuning strategy.

\item We prove the convergence of the proposed algorithm and show that the convergence rate is $O\left(\frac{1}{k}\right)$, where $k$ is the
iteration counter. Due to the automatic update of the algorithmic parameters and the low computational complexity per iteration, the proposed
algorithm performs better than some related existing decomposition algorithms from  the literature in terms of computational time.

\item An extension to a switching strategy is also presented.  This algorithm updates simultaneously two smoothness parameters at each iteration and makes use
of the inexactness of the gradients of the smoothed dual function.
\end{enumerate}

Let us emphasize the following points of the contribution.
The first algorithm proposed in this paper consists of two dual steps and one primal step per iteration. This requires  solving the primal subproblems in
parallel only \textit{once} but needs one more dual step. Because the dual step corresponds only to a simple matrix-vector multiplication, the computational
cost of the proposed algorithm is significantly reduced compared to some existing decomposition methods in the literature. Moreover, since solving the primal
subproblems exactly is only \textit{conceptual} (except existing a closed form solution), we propose an inexact algorithm which allows one to solve these
problems up to a given accuracy. The accuracies
of solving the primal subproblems are adaptively chosen such that the convergence of the whole algorithm is preserved. The parameters in the algorithm are
updated automatically based on an analysis of the iteration scheme.  This is different from augmented Lagrangian approaches \cite{Boyd2011,Fukushima1992} where
we need to find an appropriate way to tune the penalty parameter in each practical situation.

In the switching variant, apart from the inexactness, this algorithm allows one to update simultaneously both
smoothness parameters at each iteration. The advantage of this algorithm compared to the first one is that it takes into account the convergence behavior of the
primal and dual steps which accelerates the convergence of the algorithm in some practical situations. Since both algorithms are primal-dual methods, we not
only obtain an approximate solution of the dual problem but also an approximate solution of the original problem \eqref{eq:primal_cvx_prob} without any
auxiliary computation.

%+ Paper organization.
\vskip0.1cm
\noindent\textbf{Paper outline. }
The rest of this paper is organized as follows. In the next section, we briefly  describe the Lagrangian dual decomposition method for separable convex
optimization. Section \ref{sec:smoothing} mainly presents the smoothing technique via prox-functions as well as the inexact excessive gap condition.
Section \ref{sec:decomp_alg} builds a new inexact algorithm called  \textit{inexact decomposition algorithm with two dual steps} (Algorithm \ref{alg:A1}). The
convergence rate of this algorithm is established.  Section \ref{sec:inexact_switching_alg} presents an inexact switching variant of Algorithm \ref{alg:A1}
proposed in Section
\ref{sec:decomp_alg}. Numerical examples are presented in Section \ref{sec:num_results} to examine the performance of the proposed algorithms and a numerical
comparison is made. In order to make the paper more compact, we move some the technical proofs to Appendix \ref{sec:proofs}.

%+ Notation.
\vskip 0.1cm
\noindent\textbf{Notation. }
Throughout  the paper, we shall consider the Euclidean space $\mathbb{R}^n$ endowed with an inner product $x^Ty$ for $x, y \in \mathbb{R}^n$ and the norm
$\norm{x} := \sqrt{x^Tx}$. For a given matrix $A \in \mathbb{R}^{m\times n}$, the spectral norm  $\norm{A}$ is used in the paper. The notation $x = (x_1,
\dots, x_M)$ represents a column vector in $\mathbb{R}^n$, where $x_i$ is a subvector in $\mathbb{R}^{n_i}$ and  $i=1,\dots, M$. We denote by $\mathbb{R}_{+}$
and $\mathbb{R}_{++}$ the sets of nonnegative and positive real numbers, respectively. We also use $\partial{f}$ for the subdifferential of a convex function
$f$. For a given convex set $X$ in $\mathbb{R}^n$, we denote $\mathrm{ri}(X)$ the relative interior of $X$, see, e.g. \cite{Rockafellar1970}.

%%%%%%%%%%%%%%%%%%%%%%%%%%%%%%%%%%%%%%%%%%%%%%%%%%%%%%%%%%%%%%%%%%%%%%%%%%%%%%%%%%%%%%%%%%%%%%%
%+ 2. Lagrangian dual decomposition and smoothing techniques.
%%%%%%%%%%%%%%%%%%%%%%%%%%%%%%%%%%%%%%%%%%%%%%%%%%%%%%%%%%%%%%%%%%%%%%%%%%%%%%%%%%%%%%%%%%%%%%%
\section{Lagrangian dual decomposition}\label{Lagrangian_decomp}
In this section, we briefly describe the Lagrangian dual decomposition technique in convex optimization, see, e.g. \cite{Bertsekas1989}.
Let $x := (x_1, \dots, x_M)$ be a  vector and
$A := [A_1,\dots, A_M]$ be a matrix formed from $M$ components $x_i$ and $A_i$, respectively. Let $b := \sum_{i=1}^Mb_i$ and $X :=
X_1\times\cdots\times X_M$.
The Lagrange function associated with the coupling constraint $Ax - b = 0$ is defined by
\begin{equation*}
\mathcal{L}(x, y) := \phi(x) + y^T(Ax - b) = \sum_{i=1}^M\left[\phi_i(x_i) + y^T(A_ix_i-b_i)\right],
\end{equation*}
where $y$ is the Lagrange multiplier associated with $Ax - b = 0$.
The dual problem of \eqref{eq:primal_cvx_prob} is written as
\begin{equation}\label{eq:dual_cvx_prob}
g^{*} := \max_{y\in\mathbb{R}^m}g(y),
\end{equation}
where $g(\cdot)$ is the dual function defined by
\begin{equation}\label{eq:dual_fun}
g(y) := \min_{x\in X}\mathcal{L}(x, y) = \min_{x\in X}\{\phi(x) + y^T(Ax-b)\}.
\end{equation}
Note that the dual function $g$ can be computed in \textit{parallel} for each component $x_i$ as
\begin{eqnarray}\label{eq:gi_fun}
g(y) := \sum_{i=1}^Mg_i(y), ~~~\mathrm{where}~ g_i(y) \!:=\! \min_{x_i\in X_i}\!\!\left\{ \phi_i(x_i) + y^T(A_ix_i - b_i)\right\}, ~~i=1,\dots, M.
\end{eqnarray}
We denote $x_i^{*}(y)$ a solution of the minimization problem in \eqref{eq:gi_fun}. Consequently, $x^{*}(y) := (x_1^{*}(y), \dots, x_M^{*}(y))$ is a
solution of \eqref{eq:dual_fun}.
It is well-known that $g$ is concave and  the dual problem \eqref{eq:dual_cvx_prob} is convex but nondifferentiable in general.

Throughout the paper, we assume that the following assumptions hold \cite{Ruszczynski1995}.
%+ Assumption A.1.
\begin{assumption}\label{as:A1}
The solution set $X^{*}$ of \eqref{eq:primal_cvx_prob} is nonempty and either $X$ is polyhedral or the Slater constraint qualification condition for problem
\eqref{eq:primal_cvx_prob} holds, i.e.
\begin{equation}\label{eq:slater_cond}
\left\{ x\in\mathbb{R}^n ~|~ Ax - b = 0 \right\} ~\cap~ \mathrm{ri}(X) \neq \emptyset.
\end{equation}
For each $i \in \left\{1, 2, \dots, M\right\}$, $\phi_i$ is proper, lower semicontinuous and convex in $\mathbb{R}^{n_i}$.
\end{assumption}

If $X$ is convex and bounded then $X^{*}$ is also convex and bounded. Note that the objective function $\phi$ is not necessarily smooth. For example, $\phi(x)
:= \norm{x}_1 = \sum_{i=1}^n|x_{(i)}|$,  which is nonsmooth and separable, can be handled in our framework. Under Assumption \aref{as:A1}, the solution set
$Y^{*}$ of the dual problem \eqref{eq:dual_cvx_prob} is nonempty and bounded. Moreover, \textit{strong duality condition} holds, i.e. for all $(x^{*},
y^{*})\in X^{*}\times Y^{*}$ we have $\phi^{*} = \phi(x^{*}) = g(y^{*}) = g^{*}$.
If strong duality holds then we can refer to $g^{*}$ or $\phi^{*}$ as the \textit{primal-dual optimal value}.

%%%%%%%%%%%%%%%%%%%%%%%%%%%%%%%%%%%%%%%%%%%%%%%%%%%%%%%%%%%%%%%%%%%%%%%%%%%%%%%%%%%%%%%%%%%%%%%
%+ 3. Smoothing technique via barrier functions.
%%%%%%%%%%%%%%%%%%%%%%%%%%%%%%%%%%%%%%%%%%%%%%%%%%%%%%%%%%%%%%%%%%%%%%%%%%%%%%%%%%%%%%%%%%%%%%%
\section{Smoothing via prox-functions}\label{sec:smoothing}
Since the dual function $g$ is in general nonsmooth, one can  apply smoothing techniques to approximate $g$ up to a desired accuracy. In this section, we
propose to use a smoothing technique via proximity functions proposed in \cite{Nesterov2005}.

%+ ******************************************************************
%+ 3.1. Proximity function.
%+ ******************************************************************
\vskip0.1cm
\noindent\textbf{3.1. Proximity functions. }
Let $C$ be a nonempty, closed and  convex set in $\mathbb{R}^n$. We consider a nonnegative, continuous and strongly convex function $p_C : C \to \mathbb{R}_{+}$
with a convexity parameter $\sigma_p > 0$. As usual, we call $p_C$ a \textit{proximity function} (prox-function) associated with the convex set $C$. Let
\begin{equation}\label{eq:D_quantity}
p^{*}_C := \min_{x\in C}p_C(x) ~~~\mathrm{and}~~~ D_C := \sup_{x\in C}p_C(x).
\end{equation}
Since $p_C$ is strongly convex, there exists a unique  point $x^c\in C$ such that $p^{*}_C = p_C(x^c)$. The point $x^c$ is called the \textit{proximity center}
of $C$ w.r.t. $p_C$.  Moreover, if $C$ is bounded then $0 \leq p^{*}_C \leq D_C < +\infty$. Without loss of generality, we can assume that $p_C^{*} > 0$.
Otherwise, we can shift this function as $\bar{p}_C(x) := p_C(x) + r_0$, where $r_0 + p_C^{*} > 0$.

% Remark 3.1.
\begin{remark}\label{re:quad_prox}
We note that the simplest prox-function is the quadratic form $p_C(x) := \frac{\sigma_p}{2}\norm{x-x^c}^2 + r$, where $r > 0$,
$\sigma_p > 0$ and $x^c\in C$ are given. If the set $C$ has a specific structure then one can choose an appropriate prox-function
that  captures better the structure of $C$ than  the quadratic prox-function. For example, if $C$ is a standard simplex, one can
choose the entropy prox-function as mentioned in \cite{Nesterov2005}. If $C$ has no specific structure, then we can use the quadratic prox-function given above.
Consequently, the convex problem generated using quadratic prox-functions reduces in some cases to a simple optimization problem, so that its solution
can be computed numerically very efficient.
\end{remark}

%+ ******************************************************************
%+ 3.2. Smoothed approximations.
%+ ******************************************************************
\vskip0.1cm
\noindent\textbf{3.2. Smoothed approximations.}
In order to build smoothed approximations of the objective function $\phi$ and the dual function $g$ in the framework of the primal-dual smoothing technique
proposed in \cite{Nesterov2005a}, we make the following assumption.

%+ Assumption A.2.
\begin{assumption}\label{as:A2}
Each feasible set $X_i$ admits a  prox-function $p_{X_i}$  with a
convexity parameter $\sigma_i > 0$ and the proximity center $x_i^c$.
Further, we assume
\begin{equation*}
0 <
p_{X_i}^{*} := \displaystyle\min_{x\in X_i}p_{X_i}(x_i) \leq D_{X_i} := \displaystyle\sup_{x_i\in X_i}p_{X_i}(x_i) < +\infty, ~~i=1,\dots, M.
\end{equation*}
\end{assumption}

\noindent
If $X_i$ is bounded for $i=1,\dots,M$, then Assumption \ref{as:A2} is satisfied. If $X_i$ is unbounded, then we can assume that our sample points generated by
the proposed algorithms are bounded. In this case, we can restrict the feasible set of problem \eqref{eq:primal_cvx_prob} on $X\cap C$, where $C$ is a given
compact set which contains the sample points and the desired solutions of \eqref{eq:primal_cvx_prob}.

We denote by
\begin{equation}\label{eq:D_X_p_X}
p_X(x) := \sum_{i=1}^Mp_{X_i}(x_i), ~ p_X^{*} := \sum_{i=1}^Mp_{X_i}^{*} > 0, ~\mathrm{and} ~ D_X := \sum_{i=1}^MD_{X_i} < +\infty.
\end{equation}
Since $\phi_i$ is not necessarily strictly convex, the function $g_i$ defined by \eqref{eq:gi_fun} may not be differentiable. We consider the following
function
\begin{eqnarray}\label{eq:smoothed_gi}
g(y; \beta_1) := \sum_{i=1}^Mg_i(y; \beta_1), ~\textrm{where}~
g_i(y; \beta_1) := \min_{x_{i}\in X_{i}}\left\{\phi_i(x_{i}) +
y^T(A_{i}x_{i} - b_i) + \beta_1p_{X_i}(x_{i})\right\},
\end{eqnarray}
for $i=1,\dots, M$ and $\beta_1 > 0$ is a given \textit{smoothness parameter}.
We denote $x_i^{*}(y; \beta_1)$  the unique solution of \eqref{eq:smoothed_gi}, i.e.
\begin{eqnarray}\label{eq:smooth_dual_sol}
&x^{*}_i(y; \beta_1) := \textrm{arg}\!\!\displaystyle\min_{x_{i}\in X_{i}}\left\{\phi_i(x_{i}) + y^T(A_{i}x_{i} - b_i) +
\beta_1p_{X_i}(x_{i})\right\}, ~i=1, \dots, M,
\end{eqnarray}
and $x^{*}(y;\beta_1) := (x^{*}_1(y;\beta_1), \dots, x^{*}_M(y; \beta_1))$.
We call each minimization problem in \eqref{eq:smoothed_gi} a \textit{primal subproblem}.
Note that we can use different smoothness parameters $\beta_1^{i}$ in \eqref{eq:smoothed_gi} for each $i\in\{1,\dots, M\}$.
First, we recall the following properties of $g(\cdot; \beta_1)$, see \cite{Nesterov2005}.

%+ Lemma 3.1.
\begin{lemma}\label{le:dual_smoothed_estimates}
For any $\beta_1 > 0$, the function $g(\cdot; \beta_1)$ defined by \eqref{eq:smoothed_gi} is concave and differentiable. The gradient of $g(\cdot;\beta_1)$ is
given by
$\nabla_y{g}(y; \beta_1) := Ax^{*}(y; \beta_1) - b$  which is
Lipschitz continuous with a Lipschitz constant
\begin{equation}\label{eq:Lg_beta}
L^{g}(\beta_1) := \frac{1}{\beta_1}\sum_{i=1}^M\frac{\norm{A_i}^2}{\sigma_i}.
\end{equation}
Moreover, we have the following estimates:
\begin{equation}\label{eq:g_approx}
g(y; \beta_1) - \beta_1D_X \leq g(y) \leq g(y; \beta_1),
\end{equation}
and
\begin{equation}\label{eq:g_Lipschitz_bound}
g(\tilde{y}; \beta_1) + \nabla_yg(\tilde{y}; \beta_1)^T(y-\tilde{y}) - \frac{L^g(\beta_1)}{2}\norm{y-\tilde{y}}^2 \leq g(y; \beta_1), ~\forall
y, \tilde{y}\in\mathbb{R}^m.
\end{equation}
\end{lemma}
%+ End of the lemma.

\noindent Next, we consider the variation of the function $g(y;\cdot)$ w.r.t. the parameter $\beta_1$.

%+ Lemma 3.2.
\begin{lemma}\label{le:g_beta_1}
Let us fix $y\in\mathbb{R}^m$. The function $g(y; \cdot)$ defined by \eqref{eq:smoothed_gi} is well-defined, nondecreasing, concave and differentiable in
$\mathbb{R}_{++}$. Moreover,  the following inequality holds:
\begin{equation}\label{eq:g_beta1_Lipschitz_bound}
g(y; \beta_1) \leq g(y; \tilde{\beta}_1) + (\beta_1-\tilde{\beta}_1)p_X(x^{*}(y; \tilde{\beta}_1)), ~\beta_1,
\tilde{\beta}_1\in\mathbb{R}_{++},
\end{equation}
where $x^{*}(y; \tilde{\beta}_1)$ is defined by \eqref{eq:smooth_dual_sol}.
\end{lemma}

%+ Proof of Lemma 3.2.
\begin{proof}
Since $g = \sum_{i=1}^Mg_i$ and $p_X = \sum_{i=1}^Mp_{X_i}$, it  is sufficient to prove the inequality \eqref{eq:g_beta1_Lipschitz_bound} for $g_i(y; \cdot)$,
with $i=1,\dots, M$.
Let us fix $y\in\mathbb{R}^m$ and $i\in\{1,\dots, M\}$. We define $\phi_i(x; \beta_1) := \phi_i(x) + y^T(A_ix_i-b_i) + \beta_1p_{X_i}(x_i)$ a function of
two joint variables $x_i$ and $\beta_1$.
Since $\phi_i(\cdot;\cdot)$ is strongly convex w.r.t. $x_i$ and linear w.r.t. $\beta_1$, $g_i(y; \beta_1) := \displaystyle\min_{x_i\in X_i}\phi_i(x_i;\beta_1)$
is well-defined and concave w.r.t. $\beta_1$. Moreover, it is differentiable w.r.t. $\beta_1$ and $\nabla_{\beta_1}g(y; \beta_1) = p_{X_i}(x^{*}_i(y; \beta_1))
\geq 0$, where $x^{*}_i(y; \beta_1)$ is defined in \eqref{eq:smooth_dual_sol}.
Hence, $g_i(y; \cdot)$ is nonincreasing. By using the concavity of $g_i(y; \cdot)$ we have
\begin{align*}
g_i(y; \beta_1) \leq g_i(y; \tilde{\beta}_1) + (\beta_1-\tilde{\beta}_1)\nabla_{\beta_1}g_i(y; \tilde{\beta}_1) = g_i(y;\tilde{\beta}_1) +
(\beta_1-\tilde{\beta}_1)p_{X_i}(x^{*}_i(y; \tilde{\beta}_1)). \nonumber
\end{align*}
By summing up the last inequality from $i=1$ to $M$ and then using \eqref{eq:D_X_p_X} we obtain \eqref{eq:g_beta1_Lipschitz_bound}.
\Eproof
\end{proof}
%+ End of the proof.

Finally, we consider a smooth approximation to $\phi$.
Let $p_Y(y) := \frac{1}{2}\norm{y}^2$ be a prox-function defined in $\mathbb{R}^m$ with a convexity parameter $\sigma_{p_Y} = 1 > 0$. It is obvious
that
the proximity center of $p_Y$ is $y^c := 0^m\in\mathbb{R}^m$. We define the following function on $X$:
\begin{equation}\label{eq:psi_fun}
\psi(x; \beta_2) := \max_{y\in\mathbb{R}^m}\left\{ (Ax - b)^Ty - \frac{\beta_2}{2}\norm{y}^2 \right\},
\end{equation}
where $\beta_2 > 0$ is the second smoothness parameter.
We denote by $y^{*}(x;\beta_2)$ the solution of \eqref{eq:psi_fun}.
From \eqref{eq:psi_fun}, we see that $\psi(x;\beta_2)$ and $y^{*}(x;\beta_2)$ can be
computed explicitly as
\begin{equation}\label{eq:explicit_psi_fun}
\psi(x; \beta_2) := \frac{1}{2\beta_2}\norm{Ax - b}^2  ~~~\mathrm{and}~~~ y^{*}(x;\beta_2) := \frac{1}{\beta_2}(Ax - b).
\end{equation}
It clear that $\psi(x;\beta_2) \geq 0$ for all $x\in X$.
Now, we define the function $f(x; \beta_2)$ as
\begin{equation}\label{eq:smoothed_phi}
f(x; \beta_2) := \phi(x) + \psi(x; \beta_2).
\end{equation}
Then, $f(x; \beta_2)$ is exactly a quadratic penalty function of \eqref{eq:primal_cvx_prob}. The following lemma shows that $f(\cdot; \beta_2)$ is an
approximation of $\phi$.

%+ Lemma 3.3.
\begin{lemma}\label{le:primal_smoothed_estimates}
The function $\psi$ defined by \eqref{eq:psi_fun} satisfies the following estimate:
\begin{eqnarray}\label{eq:psi_Lipschitz_bound}
\psi(x; \beta_2) \leq \psi(\tilde{x}; \beta_2) + \nabla_x\psi(\tilde{x}; \beta_2)^T(x -\tilde{x}) + \sum_{i=1}^M\frac{L_i^{\psi}(\beta_2)}{2}\norm{x_i
-\tilde{x}_i}^2, ~~\forall x, \tilde{x}\in X,
\end{eqnarray}
where $L^{\psi}_i(\beta_2) :=  \frac{M\norm{A_i}^2}{\beta_2}$.
Moreover, the function $f$ defined by \eqref{eq:smoothed_phi} satisfies
\begin{equation}\label{eq:psi_approx}
f(x; \beta_2) - \frac{1}{2\beta_2}\norm{Ax - b}^2 = f(x; \beta_2) - \psi(x; \beta_2) = \phi(x) \leq f(x;  \beta_2), ~\forall x\in X.
\end{equation}
\end{lemma}

%+ The proof of Lemma 3.3.
\begin{proof}
By the definition of $\psi$, we have $\psi(x; \beta_2) - \psi(\tilde{x}; \beta_2) -  \nabla_x\psi(\tilde{x}; \beta_2)^T(x -\tilde{x}) =
\frac{1}{2\beta_2}\norm{A(x-\tilde{x})}^2$. Thus \eqref{eq:psi_Lipschitz_bound} follows from this equality by applying some elementary
inequalities. The bounds \eqref{eq:psi_approx} follow directly from the definition \eqref{eq:smoothed_phi} of $f$.
\Eproof
\end{proof}
%+ End of the proof.

%+ ******************************************************************
%+ 3.3. Inexact solution of the subproblems.
%+ ******************************************************************
\vskip0.1cm \noindent\textbf{3.3. Inexact solutions of the primal
subproblem.} Regarding the primal subproblem \eqref{eq:smoothed_gi},
if the objective function $\phi_i$ has a specific form, e.g.
univariate functions,  then we can solve this problem analytically
(exactly) to obtain a \textit{closed form} solution. A simple
example of such function is $\phi_i(x_i) = \abs{x_i}$. However, in
most practical problems, solving  the primal subproblem
\eqref{eq:smoothed_gi} exactly is only conceptual. In practice, we
only solve this problem up to a given accuracy. In other words, for
each $i\in\{1,2,\dots, M\}$, the solution $x^{*}_i(y; \beta_1)$ in
\eqref{eq:smoothed_gi} is approximated by
\begin{eqnarray}\label{eq:ap_smoothed_gi_sol}
\tilde{x}^{*}_i(y; \beta_1) :\approx \textrm{arg}\!\!\!\displaystyle\min_{x_{i}\in X_{i}}\left\{\phi_i(x_{i}) + y^T(A_{i}x_{i} - b_i) +
\beta_1p_{X_i}(x_{i})\right\},
\end{eqnarray}
in the sense of the following definition.

%+ Definition 3.1.
\begin{definition}\label{de:inexactness}
We say that the point $\tilde{x}_i^{*}(y; \beta_1)$ approximates $x^{*}_i(y; \beta_1)$ defined by \eqref{eq:smooth_dual_sol} up to a given accuracy
$\varepsilon_i\geq 0$ if:
\begin{itemize}
\item[$\mathrm{a)}$] it is feasible to $X_i$, i.e. $\tilde{x}_i^{*}(y; \beta_1)\in X_i$;
\item[$\mathrm{b)}$] $\tilde{x}_i^{*}(y; \beta_1)$ satisfies the condition:
\begin{equation}\label{eq:inexactness_xi}
0 \leq h_i(\tilde{x}_i^{*}(y; \beta_1); y, \beta_1) - h_i(x^{*}_i(y; \beta_1); y, \beta_1) \leq \frac{\beta_1\sigma_i}{2}\varepsilon_i^2,
\end{equation}
\end{itemize}
where $h_i(x_i; y, \beta_1) := \phi_i(x_i) + y^T(A_ix_i - b_i) + \beta_1p_{X_i}(x_i)$.
\end{definition}

In practice, for a given accuracy  $\varepsilon_i > 0$, we can check
whether the conditions of Definition \ref{de:inexactness} are
satisfied by applying classical convex optimization algorithms, e.g.
(sub)gradient or interior-point algorithms \cite{Nesterov2004}.

Since $h_i(\cdot;y, \beta_1)$ is strongly convex with a convexity parameter  $\beta_1\sigma_i > 0$, we have
\begin{eqnarray}\label{eq:ap_xi_star_sol}
\frac{\beta_1\sigma_i}{2}\norm{\tilde{x}^{*}_i(y;\beta_1) - x^{*}_i(y;\beta_1)}^2 \leq h_i(\tilde{x}^{*}_i(y;\beta_1); y,
\beta_1) - h_i(x^{*}_i(y;\beta_1); y, \beta_1)  \leq \frac{\beta_1\sigma_i}{2}\varepsilon_i^2,
\end{eqnarray}
where $h_i(\cdot; y, \beta_1)$ is defined as in Definition \ref{de:inexactness}.
Consequently, we have: $\norm{\tilde{x}^{*}_i(y; \beta_1) - x^{*}_i(y;\beta_1)} \leq \varepsilon_i$ for $i=1,\dots, M$.
Let $\tilde{x}^{*}(y; \beta_1) := (\tilde{x}^{*}_1(y; \beta_1),\dots, \tilde{x}_M^{*}(y; \beta_1))$ and
\begin{equation}\label{eq:app_deriv_g}
\widetilde{\nabla}_yg(y; \beta_1) := A\tilde{x}^{*}(y; \beta_1) - b.
\end{equation}
The quantity $\widetilde{\nabla}_yg(\cdot;\beta_1)$ can be referred to as  an approximation of the gradient $\nabla_y{g}(\cdot; \beta_1)$ defined in Lemma
\ref{le:dual_smoothed_estimates}. If we denote by $\mathbf{\varepsilon} := (\varepsilon_1,\dots,\varepsilon_M)^T$ the vector of accuracy levels then we can
easily estimate
\begin{equation}\label{eq:app_deriv_g_est}
\norm{\widetilde{\nabla}_yg(y; \beta_1) - \nabla_yg(y; \beta_1)} = \norm{A(\tilde{x}^{*}(y; \beta_1) - x^{*}(y; \beta_1))}
\leq \norm{A}\norm{\mathbf{\varepsilon}}.
\end{equation}

%+ ******************************************************************
%+ 3.4. Excessive gap technique.
%+ ******************************************************************
\vskip0.1cm
\noindent\textbf{3.4. Inexact excessive gap condition.}
Since problem \eqref{eq:primal_cvx_prob} is convex, under Assumption \aref{as:A1}  strong duality holds. The aim is to generate a primal-dual sequence
$\{(\bar{x}^k, \bar{y}^k)\}_{k\geq 0}$ such that for a sufficiently large $k$ the point $\bar{x}^k$ is approximately feasible to
\eqref{eq:primal_cvx_prob}, i.e. $\norm{A\bar{x}-b} \leq \varepsilon_p$, and the primal-dual gap satisfies $\abs{\phi(\bar{x}^k) - g(\bar{y}^k)} \leq
\varepsilon_d$ for given tolerances $\varepsilon_d \geq 0$ and $\varepsilon_p\geq 0$.

The algorithm designed  below will employ the approximate functions \eqref{eq:smoothed_gi}-\eqref{eq:smoothed_phi} to solve the primal-dual problems
\eqref{eq:primal_cvx_prob}-\eqref{eq:dual_cvx_prob}.
First, we modify the excessive gap condition introduced by Nesterov in \cite{Nesterov2005a} to the inexact case in the following definition.

%+ Definition 3.2.
\begin{definition}\label{de:exessive_gap_cond}
A point $(\bar{x}, \bar{y})\in X\times\mathbb{R}^m$   satisfies the \textit{inexact excessive gap} ($\delta$-\textit{excessive gap}) condition w.r.t. $\beta_1 >
0$ and $\beta_2 > 0$ and a given accuracy $\delta \geq 0$ if
\begin{equation}\label{eq:excessive_gap_inexact}
f(\bar{x}; \beta_2) \leq g(\bar{y}; \beta_1) + \delta.
\end{equation}
\end{definition}
If $\delta = 0$ then \eqref{eq:excessive_gap_inexact} reduces to the exact excessive gap  condition considered in \cite{Nesterov2005a}.

The following lemma provides an upper bound estimate for the primal-dual gap and the feasibility gap of problem \eqref{eq:primal_cvx_prob}.

%+ Lemma 3.4.
\begin{lemma}\label{le:excessive_gap}
Suppose that $(\bar{x}, \bar{y}) \in X\times\mathbb{R}^m$ satisfies the $\delta$-excessive gap condition \eqref{eq:excessive_gap_inexact}. Then for any
$y^{*}\in Y^{*}$, we have
\begin{eqnarray}
\mathcal{F}(\bar{x}) := \norm{A\bar{x} - b} &&{\!\!\!}\leq \beta_2\Big[\norm{y^{*}} + \Big(\norm{y^{*}}^2 + \frac{2\beta_1}{\beta_2}D_X
+ \frac{2\delta}{\beta_2}\Big)^{1/2}\Big], \label{eq:feasibility}\\
[-1.5ex]
\textrm{and} {~~~~~~~~~~~~~~~~~~~~~~~~~~~}&& \nonumber\\
[-1.5ex]
-\norm{y^{*}}\mathcal{F}(\bar{x}) &&{\!\!\!}\leq \phi(\bar{x}) - g(\bar{y}) \leq \delta \!+\! \beta_1D_X - \frac{\mathcal{F}(\bar{x})^2}{2\beta_2} \leq \delta +
\beta_1D_X. \label{eq:dual_gap}
\end{eqnarray}
\end{lemma}

%+ Proof of Lemma 3.4.
\begin{proof}
From the estimates \eqref{eq:g_approx} and \eqref{eq:psi_approx} we have
\begin{equation}
\phi(\bar{x}) - g(\bar{y}) \leq f(\bar{x}; \beta_2) - g(\bar{y}; \beta_1) + \beta_1D_X - \frac{1}{2\beta_2}\norm{A\bar{x} - b}^2.
\end{equation}
Then, by using \eqref{eq:excessive_gap_inexact}, the last inequality implies the right-hand side of \eqref{eq:dual_gap}.
Next, for a given $y^{*}\in Y^{*}$ we  have $g(\bar{y})  \leq \max_yg(y) = g(y^{*}) = \min_{x\in X}\left\{\phi(x) +
(y^{*})^T(Ax-b)\right\} \leq \phi(\bar{x}) +
(y^{*})^T(A\bar{x}-b) \leq \phi(\bar{x}) + \norm{y^{*}}\norm{A\bar{x} - b}$. Thus we obtain the left-hand side of \eqref{eq:dual_gap}.
Finally, the estimate \eqref{eq:feasibility} follows from \eqref{eq:dual_gap} after a few simple calculations.
\Eproof
\end{proof}
%+ End of the proof.

Let us define $R_{Y^{*}} :=  \max_{y^{*}\in Y^{*}}\norm{y^{*}}$ the diameter of $Y^{*}$.
Since $Y^{*}$ is bounded, we have $0 \leq R_{Y^{*}} < +\infty$. The estimates \eqref{eq:feasibility} and \eqref{eq:dual_gap} can be simplified as
\begin{eqnarray}\label{eq:simplified_feasibility_gap}
\mathcal{F}(\bar{x}) \!\leq\! 2\beta_2R_{Y^{*}} \!+\! \sqrt{2(\beta_1\beta_2D_X \!+\! \delta\beta_2)} ~~~ \textrm{and}~~
-R_{Y^{*}}\mathcal{F}(\bar{x}) \!\leq\! \phi(\bar{x}) \!-\! g(\bar{y}) \leq \beta_1D_X \!+\! \delta.
\end{eqnarray}

%%%%%%%%%%%%%%%%%%%%%%%%%%%%%%%%%%%%%%%%%%%%%%%%%%%%%%%%%%%%%%%%%%%%%%%%%%%%%%%%%%%%%%%%%%%%%%%%%%%%
%+ 4. The Decomposition Algorithm via Dual Step Update
%%%%%%%%%%%%%%%%%%%%%%%%%%%%%%%%%%%%%%%%%%%%%%%%%%%%%%%%%%%%%%%%%%%%%%%%%%%%%%%%%%%%%%%%%%%%%%%%%%%%
\section{Inexact decomposition algorithm with one primal step and two dual steps}\label{sec:decomp_alg}
In this section we first show that, for a given $\delta_0 \geq 0$, there exists a point $(\bar{x}^0, \bar{y}^0)\in X\times\mathbb{R}^m$ such that the
condition \eqref{eq:excessive_gap_inexact} is
satisfied. Then, we propose a decomposition scheme to update successively a sequence $\{(\bar{x}^k,\bar{y}^k)\}_{k\geq 0}$ that
maintains the condition \eqref{eq:excessive_gap_inexact} while it drives the sequences of smoothness parameters $\{\beta_1^k\}_{k\geq 0}$ and
$\{\beta_2^k\}_{k\geq 0}$ to zero.

Let us introduce the following quantities
\begin{equation}\label{eq:accuracy_constants}
\begin{cases}
\varepsilon_{[\sigma]} &:= \left[\sum_{i=1}^M\sigma_i\varepsilon_i^2\right]^{1/2},\\
D_{\sigma} &:= \left[2\sum_{i=1}^M\frac{D_{X_i}}{\sigma_i}\right]^{1/2},\\
C_d                  &:=  \norm{A}^2D_{\sigma} + \norm{A^T(Ax^c - b)},\\
L_{\mathrm{A}}     &:=  M\max\set{\frac{\norm{A_i}^2}{\sigma_i}~|~1\leq i\leq M}.
\end{cases}
\end{equation}
From \eqref{eq:accuracy_constants}  we see that the constant $C_d$ depends on the data of the problem (i.e. $A$, $D_{X}$, $\sigma$, $b$ and $x^c$). Moreover,
$\varepsilon_{[1]} = \norm{\mathbf{\varepsilon}}$. If we choose the accuracy  $\varepsilon_i = \hat{\varepsilon} \geq 0$ for all $i=1,\dots, M$, then
$\varepsilon_{[1]} = \sqrt{M}\hat{\varepsilon}$ and $\varepsilon_{[\sigma]} =
[\sum_{i=1}^M\sigma_i]^{1/2}\hat{\varepsilon}$.

%+ ******************************************************************
%+ 4.1. Finding a starting point.
%+ ******************************************************************
\vskip0.1cm
\noindent\textbf{4.1. Finding a starting point.}
For a given a positive value $\beta_1 > 0$, let $(\bar{x}^0, \bar{y}^0)$ be a point in $X\times\mathbb{R}^m$ computed as
\begin{equation}\label{eq:start_point}
\begin{cases}
\bar{x}^0 &:= \tilde{x}^{*}(0^m; \beta_1),\\
\bar{y}^0 &:=  L^{g}(\beta_1)^{-1}(A\bar{x}^0 - b),
\end{cases}
\end{equation}
where $0^m\in\mathbf{R}^m$ is the origin and $\tilde{x}^{*}(0^m; \beta_1)$ is defined by \eqref{eq:ap_smoothed_gi_sol} and $L^g(\beta_1)$ is given by
\eqref{eq:Lg_beta}. The following lemma shows that $(\bar{x}^0,
\bar{y}^0)$ satisfies the $\delta_0$-excessive gap condition \eqref{eq:excessive_gap_inexact}. The proof of this lemma
is given later in Appendix \ref{sec:proofs}.

%+ Lemma 4.1.
\begin{lemma}\label{le:init_point}
The point $(\bar{x}^0, \bar{y}^0)\in X\times\mathbb{R}^m$  generated by \eqref{eq:start_point} satisfies the $\delta_0$-excessive gap condition
\eqref{eq:excessive_gap_inexact} w.r.t. $\beta_1$ and $\beta_2$ provided that
\begin{equation}\label{eq:init_point_cond}
\beta_1\beta_2 \geq L_{\mathrm{A}},
\end{equation}
where $\delta_0 := \beta_1\left(\frac{C_d}{\bar{L}_{\mathrm{A}}}\varepsilon_{[1]} + \frac{1}{2}\varepsilon^2_{[\sigma]}\right) \geq 0$.
\end{lemma}

Note that if we use $x^{*}(0^m; \beta_1)$ instead of $\tilde{x}^{*}(0^m; \beta_1)$ into \eqref{eq:start_point}, i.e. the exact solution $x^{*}(0^m;
\beta_1)$ is used, then $(\bar{x}^0, \bar{y}^0)$ satisfies the $0$-excessive gap condition \eqref{eq:excessive_gap_inexact}.

%+ ******************************************************************
%+ 4.2. The algorithmic scheme
%+ ******************************************************************
\vskip0.1cm
\noindent\textbf{4.2. The inexact main iteration with one primal step and two dual steps.}
Let  us assume that $(\bar{x}, \bar{y})$ is a given point in $X\times\mathbb{R}^m$ that satisfies the $\delta$-excessive gap condition
\eqref{eq:excessive_gap_inexact} w.r.t. $\beta_1$, $\beta_2$ and $\delta$. The aim is to compute a new point $(\bar{x}^{+}, \bar{y}^{+})$ such that the
condition \eqref{eq:excessive_gap_inexact} holds for the new values $\beta_1^{+}$, $\beta_2^{+}$ and $\delta_{+}$ with $\beta_1^{+} < \beta_1$, $\beta_2^{+}
< \beta_2$ and $\delta_{+}\leq \delta$.

First, for a given $y$ and $\beta_1 > 0$, we define the following mapping
\begin{equation*}
\widetilde{G}^{*}(y; \beta_1) := \mathrm{arg}\!\!\max_{v\in\mathbb{R}^m}\left\{ \widetilde{\nabla}_yg(y; \beta_1)^T(v - y) - \frac{L^g(\beta_1)}{2}\norm{v
-y}^2\right\},
\end{equation*}
where $\widetilde{\nabla}_yg(y; \beta_1)$ is  defined by \eqref{eq:app_deriv_g} and $L^g(\beta_1)$ is the Lipschitz constant.
Since this maximization problem is unconstrained and convex, we can show that the quantity $\widetilde{G}^{*}(y; \hat{x}, \beta_1)$ can
be computed explicitly as
\begin{equation}\label{eq:gradient_mapping_ex}
\widetilde{G}^{*}(y; \beta_1) := y + L^g(\beta_1)^{-1}\widetilde{\nabla}_yg(y;\beta_1) = y + L^g(\beta_1)^{-1}\left[A\tilde{x}^{*}(y; \beta_1) -
b\right].
\end{equation}
Next, the main scheme to update $(\bar{x}^{+}, \bar{y}^{+})$ is presented as
\begin{equation}\label{eq:alg_scheme}
(\bar{x}^{+}, \bar{y}^{+}) := \mathcal{S}^d(\bar{x}, \bar{y}, \beta_1, \beta_2, \tau)\Leftrightarrow
\left\{
\begin{array}{lll}
\hat{y}     &:= (1-\tau)\bar{y} \!+\! \tau y^{*}(\bar{x}; \beta_2)\\
\bar{x}^{+} &:= (1-\tau)\bar{x} \!+\! \tau \tilde{x}^{*}(\hat{y}; \beta_1)\\
\bar{y}^{+} &:= \widetilde{G}^{*}(\hat{y}; \beta_1).
\end{array}
\right.
\end{equation}
Here, the smoothness parameters $\beta_1$  and $\beta_2$ and the step size $\tau \in (0, 1)$ will be appropriately updated to obtain $\beta_1^{+}$,
$\beta_2^{+}$ and $\tau_{+}$, respectively. Note that line 1 and line 3 in \eqref{eq:alg_scheme} are simply matrix-vector multiplications, which can be computed
distributively based on the structure of the coupling constraints and can be expressed as
\begin{equation*}
\hat{y} := (1-\tau)\bar{y} + \tau\beta_2^{-1}(A\bar{x}-b) ~~\textrm{and}~~ \bar{y}^{+} := \hat{y} \!+\! L^g(\beta_1)^{-1}(A\tilde{x}^{*}(\hat{y};\beta_1) - b).
\end{equation*}
Only line 2 in \eqref{eq:alg_scheme} requires one to solve $M$ convex primal subproblems up to a given accuracy. However, this  can be done in
\textit{parallel}.

Let us define
\begin{equation}\label{eq:alpha_quatity}
\alpha^{*} := \frac{p_X^{*}}{D_X} ~~~\textrm{and}~~~ \tilde{\alpha} := \frac{p_X(\tilde{x}^{*}(\hat{y};\beta_1))}{D_X}.
\end{equation}
Then, by Assumption \aref{as:A2}, we can see that $0 < \alpha^{*} \leq \tilde{\alpha} \leq 1$.
We consider an update rule for $\beta_1$ and $\beta_2$ as
\begin{equation}\label{eq:beta_update}
\beta_1^{+} := (1-\tilde{\alpha}\tau)\beta_1 ~~\textrm{and} ~~\beta_2^{+} := (1-\tau)\beta_2.
\end{equation}
In order to show that $(\bar{x}^{+}, \bar{y}^{+})$ satisfies the  $\delta_{+}$-excessive gap condition \eqref{eq:excessive_gap_inexact}, where
$\delta_{+}$ will be defined later, we define the following function
\begin{eqnarray}\label{eq:eta_func_def}
\eta(\tau, \beta_1, \beta_2, \bar{y}, \varepsilon) := \frac{\tau\beta_1}{2}\varepsilon_{[\sigma]}^2 + \left[\frac{\beta_1}{L_{\mathrm{A}}}C_d +
(1-\tau)\tau\left( \frac{C_d}{\beta_2} + \norm{A}\norm{\bar{y}}\right)\right]\varepsilon_{[1]},
\end{eqnarray}
where $\varepsilon_{[\sigma]}, C_d$ and $L_{\mathrm{A}}$ are defined in \eqref{eq:accuracy_constants}.

The next theorem provides a condition such that $(\bar{x}^{+}, \bar{y}^{+})$  generated by \eqref{eq:alg_scheme} satisfies the $\delta_{+}$-excessive gap
condition \eqref{eq:excessive_gap_inexact}.
For clarity of the exposition we move the proof of this theorem to Appendix \ref{sec:proofs}.

%+ Theorem 4.1.
\begin{theorem}\label{th:alg_scheme}
Suppose that  Assumptions \aref{as:A1} and \aref{as:A2} are satisfied.
Let $(\bar{x}, \bar{y})\in X\times\mathbb{R}^m$ be a point satisfying the $\delta$-excessive gap condition \eqref{eq:excessive_gap_inexact} w.r.t. two
values $\beta_1$ and $\beta_2$. Then if the parameter $\tau$ is chosen such that $\tau\in (0, 1)$ and
\begin{equation}\label{eq:alg_scheme_cond}
\beta_1\beta_2 \geq \frac{\tau^2}{1-\tau}L_{\mathrm{A}},
\end{equation}
then the new point $(\bar{x}^{+}, \bar{y}^{+})$ generated by \eqref{eq:alg_scheme} is in $X\times\mathbb{R}^m$ and maintains the
$\delta_{+}$-excessive gap condition \eqref{eq:excessive_gap_inexact} w.r.t two new values $\beta_1^{+}$ and $\beta^{+}_2$ defined by \eqref{eq:beta_update},
where $\delta_{+} := (1-\tau)\delta + \eta(\tau, \beta_1, \beta_2, \bar{y}, \varepsilon) $ with $\eta(\cdot)$ defined by \eqref{eq:eta_func_def}.
\end{theorem}

%+ ******************************************************************
%+ 4.3. Step size update rule
%+ ******************************************************************
\vskip0.1cm
\noindent\textbf{4.3. The step size update rule.}
Next, we show  how to update the step size $\tau \in (0, 1)$. Indeed, from \eqref{eq:alg_scheme_cond} we have
$\beta_1\beta_2 \geq \frac{\tau^2}{1-\tau}L_{\mathrm{A}}$. By combining this inequality and \eqref{eq:beta_update} we have $\beta_1^{+}\beta_2^{+} =
(1-\tau)(1-\tilde{\alpha}\tau)\beta_1\beta_2 \geq (1-\tilde{\alpha}\tau)\tau^2L_{\mathrm{A}}$. In order to ensure $\beta_1^{+}\beta_2^{+} \geq
\frac{\tau_{+}^2}{1-\tau_{+}}L_{\mathrm{A}}$ we require $(1-\tilde{\alpha}\tau)\tau^2 \geq  \frac{\tau_{+}^2}{1-\tau_{+}}$.
Since $\tau, \tau_{+} \in (0, 1)$ and $\tilde{\alpha} \in (0, 1]$, we have
\begin{equation*}
0 < \tau_{+} \leq 0.5\tau\left\{ \left[(1-\tilde{\alpha}\tau)^2\tau^2 + 4(1-\tilde{\alpha}\tau)\right]^{1/2} - (1-\tilde{\alpha}\tau)\tau\right\} < \tau.
\end{equation*}
Hence, if we  choose $\tau_{+} = 0.5\tau\left[ [(1-\tilde{\alpha}\tau)^2\tau^2 + 4(1-\tilde{\alpha}\tau)]^{1/2} - (1-\tilde{\alpha}\tau)\tau\right]$ then we
obtain the tightest rule for updating $\tau$.
Based on the  above analysis, we eventually define a sequence $\{\tau_k\}_{k\geq 0}$ as follows:
\begin{equation}\label{eq:step_size_update}
\tau_{k+1} := \frac{\tau_k}{2}\left\{\left[(1-\tilde{\alpha}_k\tau_k)^2\tau_k^2 + 4(1-\tilde{\alpha}_k\tau_k)\right]^{1/2} -
(1-\tilde{\alpha}_k\tau_k)\tau_k\right\}, ~ \forall k\geq 0,
\end{equation}
where $\tau_0 \in (0, 1)$ is given and $\tilde{\alpha}_k := p_X(\tilde{x}^{*}((\hat{y})^k;\beta_1^k))/D_X \in [\alpha^{*}, 1]$.

The following lemma provides the convergence rate of the sequence $\set{\tau_k}$, whose proof can be found in Appendix \ref{sec:proofs}.

%+ Lemma 4.2.
\begin{lemma}\label{le:choice_of_tau}
Suppose that Assumption  \aref{as:A2} is satisfied. Let $\{\tau_k\}_{k\geq 0}$ be a sequence generated by \eqref{eq:step_size_update} for a given $\tau_0$ such
that $0 < \tau_0 < [\max\{1, \alpha^{*}/(1-\alpha^{*})\}]^{-1}$. Then
\begin{equation}\label{eq:tau_bounds}
\frac{1}{k + 1/\tau_0} \leq \tau_k \leq \frac{1}{0.5(1+\alpha^{*})k + 1/\tau_0}.
\end{equation}
Moreover, the  sequences $\{\beta_1^k\}_{k\geq 0}$ and $\{\beta_2^k\}_{k\geq 0}$ generated by \eqref{eq:beta_update} satisfy
\begin{eqnarray}\label{eq:beta_bounds}
&&\frac{\gamma}{(\tau_0k + 1)^{2/(1+\alpha^{*})}}\leq \beta_1^{k+1} \leq \frac{\beta_1^0}{(\tau_0k +
1)^{\alpha^{*}}}, ~~~~~~~~~~~~ \beta_2^{k+1} \leq \frac{\beta_2^0(1-\tau_0)}{\tau_0k + 1}, \nonumber\\
[-1.5ex]\\[-1.5ex]
&& \textit{and} ~~~~~~~\beta_1^k\beta_2^{k+1} = \beta_1^0\beta_2^0\frac{(1-\tau_0)}{\tau_0^2}\tau_k^2,\nonumber
\end{eqnarray}
for a fixed positive constant $\gamma$.
\end{lemma}

% + Remark 4.1.
\begin{remark}\label{re:choice_of_tau}
The estimates  \eqref{eq:tau_bounds} of Lemma \ref{le:choice_of_tau} show that the sequence $\{\tau_k\}$ converges to zero with the convergence rate
$O(\frac{1}{k})$. Consequently, by \eqref{eq:beta_bounds}, we see that the sequence $\{\beta_1^k\beta_2^k\}$ also converges to zero with the convergence
rate $O(\frac{1}{k^2})$. From \eqref{eq:init_point_cond} and \eqref{eq:alg_scheme_cond}, we can derive an initial value $\tau_0 :=
\frac{\sqrt{5}-1}{2}$.
\end{remark}

In order to choose  the accuracy for solving the primal subproblem \eqref{eq:ap_smoothed_gi_sol}, we need to analyze the formula \eqref{eq:eta_func_def}.
Let us consider a sequence $\{\eta_k\}_{k\geq 0}$ computed by
\begin{equation*}
\eta_k := \eta(\tau_k, \beta_1^k, \beta_2^k, \bar{y}^k, \mathbf{\varepsilon}^k),
\end{equation*}
where $\eta$ is given in \eqref{eq:eta_func_def}. The sequence $\{\delta_k\}_{k\geq 0}$ defined by
\begin{equation}\label{eq:delta_sequence}
\delta_{k+1} := (1-\tau_k)\delta_k + \eta_k = \delta_k + (\eta_k - \tau_k\delta_k), ~\forall k\geq 0,
\end{equation}
where $\delta_0$ is chosen \textit{a priori}, is  nonincreasing if $\eta_k \leq \tau_k\delta_k$ for all $k\geq 0$.

%+ Lemma 4.3.
\begin{lemma}\label{le:accuracy_level_choice}
If the accuracy $\varepsilon^k_i$ at  the iteration $k$ of Algorithm \ref{alg:A1} below is chosen such that $0 \leq \varepsilon_i^k \leq
\bar{\varepsilon}^k :=
\frac{\tau_k\delta_k}{Q_k}$ for $i=1,\dots, M$, where
\begin{eqnarray}\label{eq:accuracy_level_choice}
Q_k := \frac{\tau_k\beta_1^k}{2}\sum_{i=1}^M\sigma_i  + \sqrt{M}\left[\frac{\beta_1^k}{L_{\mathrm{A}}}C_d +
(1-\tau_k)\tau_k\left( \frac{C_d}{\beta_2^k} + \norm{A}\norm{\bar{y}^k}\right)\right],
\end{eqnarray}
then the sequence $\{\delta_k\}_{k\geq 0}$ generated by \eqref{eq:delta_sequence} is nonincreasing.
\end{lemma}

%+ Proof of Lemma 4.3.
\begin{proof}
Since $0 \leq \varepsilon_i^k \leq \bar{\varepsilon}^k < 1$ for all $i=1,\dots, M$, we have $\varepsilon_{[1]}^k \leq \sqrt{M}\bar{\varepsilon}^k$ and
$(\varepsilon_{[\sigma]}^k)^2 \leq \left(\sum_{i=1}^M\sigma_i\right)(\bar{\varepsilon}^k)^2 \leq \left(\sum_{i=1}^M\sigma_i\right)\bar{\varepsilon}^k$. By
substituting these inequalities into \eqref{eq:eta_func_def} of $\eta$ and then using
\eqref{eq:accuracy_level_choice} and the notation $\eta_k = \eta(\tau_k, \beta_1^k, \beta_2^k, \bar{y}^k, \mathbf{\varepsilon}^k)$, we have
\begin{equation*}
\eta_k \leq Q_k\bar{\varepsilon}^k.
\end{equation*}
On the other hand, from \eqref{eq:delta_sequence}  we have $\delta_{k+1} = \delta_k + (\eta_k - \tau_k\delta_k)$ for all $k\geq
0$. Thus, $\{\delta_k\}_{k\geq 0}$ is nonincreasing if $\eta_k - \tau_k\delta_k \leq 0$ for all $k\geq 0$. If we choose $\bar{\varepsilon}^k$ such that
$Q_k\bar{\varepsilon}^k \leq \tau_k\delta_k$, i.e. $\bar{\varepsilon}^k \leq \frac{\tau_k\delta_k}{Q_k}$, then $\eta_k \leq \tau_k\delta_k$.
\Eproof
\end{proof}
%+ End of the proof.

From Lemma \ref{le:accuracy_level_choice} it follows that if we choose $\bar{\varepsilon}_k$  sufficiently small, then the sequence $\{(\bar{x}^k, \bar{y}^k)\}$
generated by $(\bar{x}^{k+1}, \bar{y}^{k+1}) := \mathcal{S}^d(\bar{x}^k, \bar{y}^k, \beta_1^{k+1}, \beta_2^k, \tau_k)$ maintains the $\delta_{k+1}$-excessive
gap condition \eqref{eq:excessive_gap_inexact} with $\delta_{k+1}\leq\delta_k$ for all $k$. Now, by using Lemmas \ref{le:excessive_gap} and \ref{le:init_point},
if we choose $\bar{\varepsilon}^0$ in Lemma \ref{le:accuracy_level_choice} such that $\bar{\varepsilon}^0 := \frac{\tilde{\varepsilon}}{C_0}$, where
\begin{equation}\label{eq:C0_constant}
C_0 := \beta_1^0\left(\frac{\sqrt{M}C_d}{L_{\mathrm{A}}} + \frac{1}{2}\sum_{i=1}^M\sigma_i\right),
\end{equation}
and  $\tilde{\varepsilon} \geq 0$ is a given accuracy level, then the condition \eqref{eq:excessive_gap_inexact} holds with $\delta = \tilde{\varepsilon}$.

%+ ******************************************************************
%+ 3.4. The algorithm and its complexity.
%+ ******************************************************************
\vskip0.1cm
\noindent\textbf{4.4. The algorithm and its convergence.}
Finally, we present the algorithm in detail and estimate its worst-case complexity. For simplicity of discussion, we fix the accuracy at one level
$\bar{\varepsilon}^k$  for all the primal subproblems. However, we can alternatively choose different accuracy for each subproblem by slightly
modifying  the theory presented in this paper.

%%%%%%%%%%%%%%%%%%%%%%%%%%%%%%%%%%%%%%%%%%%%%%%%%%%%%%%%%%%%%%%%%%%%%%%%%%%%%%%%%%%%%%%%%%%%%%%%
%+ Algorithm A.1.
%%%%%%%%%%%%%%%%%%%%%%%%%%%%%%%%%%%%%%%%%%%%%%%%%%%%%%%%%%%%%%%%%%%%%%%%%%%%%%%%%%%%%%%%%%%%%%%%
\begin{algorithm}{\textit{Inexact decomposition algorithm with two dual steps}}\label{alg:A1}
 \newline
\noindent\textbf{Initialization:} Perform the following steps:
\begin{itemize}
\item[]\textit{Step 1}: Provide an  accuracy level $\tilde{\varepsilon} \geq 0$ for solving \eqref{eq:smoothed_gi} and a value
$\beta^0 > 0$. Set
$\tau_0 := 0.5(\sqrt{5}-1)$, $\beta_1^0 := \beta^0$ and $\beta_2^0 := \frac{L_{\mathrm{A}}}{\beta^0}$.
\item[]\textit{Step 2}: Compute $C_0$ by \eqref{eq:C0_constant}. Set $\bar{\varepsilon}^0 := \tilde{\varepsilon}/C_0$ and $\delta_0 := \tilde{\varepsilon}$.
\item[]\textit{Step 3}: Compute $\bar{x}^0$ and $\bar{y}^0$ from \eqref{eq:start_point} as $\bar{x}^0 := \tilde{x}^{*}(0^m; \beta_1^0)$ and $\bar{y}^0
:= L^g(\beta_1^0)^{-1}(A\bar{x}^0 - b)$ up to the accuracy $\bar{\varepsilon}^0$.
\end{itemize}
\noindent\textbf{Iteration: }For $k = 0, 1, 2, \dots, k_{\max}$, perform the following steps:
\begin{itemize}
\item[]\textit{Step 1}: If a given stopping criterion is satisfied then terminate.
\item[]\textit{Step 2}: Compute $Q_k$ by \eqref{eq:accuracy_level_choice}. Set $\bar{\varepsilon}^k := \tau_k\delta_k/Q_k$ and update $\delta_{k+1} :=
(1-\tau_k)\delta_k + Q_k\bar{\varepsilon}^k$.
\item[]\textit{Step 3}: Solve the primal subproblems in \eqref{eq:smoothed_gi} \textit{in parallel} up to the accuracy $\bar{\varepsilon}^k$.
\item[]\textit{Step 4}: Compute $(\bar{x}^{k+1}, \bar{y}^{k+1}) := \mathcal{S}^d(\bar{x}^k, \bar{y}^k, \beta_1^k, \beta_2^k, \tau_k)$ by \eqref{eq:alg_scheme}.
\item[]\textit{Step 5}: Compute $\tilde{\alpha}_k := p_X(\tilde{x}^{*}(\hat{y}^k; \beta_1^k))/D_X$, where $\hat{y}^k := (1-\tau_k)\bar{y}^k +
\tau_k(\beta_2^k)^{-1}(A\bar{x}^k -b)$.
\item[]\textit{Step 6}: Update $\beta_1^{k+1} := (1-\tilde{\alpha}_k\tau_k)\beta_1^k$ and $\beta_2^{k+1} := (1-\tau_k)\beta_2^k$.
\item[]\textit{Step 7}: Update $\tau_k$ as
$\tau_{k+1} := 0.5\tau_k\left\{\left[(1-\tilde{\alpha}_k\tau_k)^2\tau_k^2 + 4(1-\tilde{\alpha}_k\tau_k)\right]^{1/2} -
(1-\tilde{\alpha}_k\tau_k)\tau_k\right\}$.
\end{itemize}
\end{algorithm}
%%%%%%%%%%%%%%%%%%%%%%%%%%%%%%%%%%%%%%%%%%%%%%%%%%%%%%%%%%%%%%%%%%%%%%%%%%%%%%%%%%%%%%%%%%%%%%%%

The stopping criterion of Algorithm \ref{alg:A1} at Step 1 will be discussed in Section \ref{sec:num_results}. The maximum number of iterations $k_{\max}$
provides a safeguard to prevent the algorithm from running to infinity.

The following theorem  provides the worst-case complexity estimate for Algorithm \ref{alg:A1} under Assumptions \aref{as:A1} and \aref{as:A2}.

%+ Theorem 4.2.
\begin{theorem}\label{th:convergence}
Suppose that Assumptions  \aref{as:A1} and \aref{as:A2} are satisfied.
Let $\{(\bar{x}^k,\bar{y}^k)\}$ be a sequence generated by Algorithm \ref{alg:A1} after $\bar{k}$ iterations. If the accuracy level $\tilde{\varepsilon}$
in Algorithm \ref{alg:A1} is chosen such that $0 \leq \tilde{\varepsilon} \leq \frac{c_0}{0.5(\sqrt{5}-1)\bar{k} + 1}$ for some positive constant $c_0$, then
the following primal-dual gap holds
\begin{equation}\label{eq:duality_gap}
-R_{Y^{*}}\mathcal{F}(\bar{x}^{\bar{k}+1}) \leq \phi(\bar{x}^{\bar{k}+1}) - g(\bar{y}^{\bar{k} + 1}) \leq \frac{(\beta^0D_X + c_0)}{[0.5(\sqrt{5}-1)\bar{k}
+1]^{\alpha^{*}}},
\end{equation}
and the feasibility gap satisfies
\begin{equation}\label{eq:feasible_gap}
\mathcal{F}(\bar{x}^{\bar{k} + 1}) = \norm{A\bar{x}^{\bar{k} + 1} - b} \leq \frac{C_f}{0.25(\sqrt{5}-1)(1+\alpha^{*})\bar{k} + 1},
\end{equation}
where $C_f :=  (3-\sqrt{5})\frac{L_{\mathrm{A}}}{\beta^0}R_{Y^{*}} + 0.5(\sqrt{5}-1)\sqrt{L_{\mathrm{A}}(D_X + c_0/\beta^0)}$ and $R_{Y^{*}}$ is
defined by \eqref{eq:simplified_feasibility_gap}.

Consequently, the  sequence $\{(\bar{x}^k, \bar{y}^k)\}_{k\geq 0}$ generated by Algorithm \ref{alg:A1} converges to a solution $(x^{*}, y^{*})$ of
the primal and dual problems \eqref{eq:primal_cvx_prob}-\eqref{eq:dual_cvx_prob} as $k\to\infty$ and $\tilde{\varepsilon}\to 0^{+}$.
\end{theorem}

%+ Proof of Theorem 4.2.
\begin{proof}
From Lemma  \ref{le:excessive_gap}, we have $\mathcal{F}(\bar{x}^{k+1}) \leq 2\beta_2^{k+1}R_{Y^{*}} +
\sqrt{2\beta_1^{k+1}\beta_2^{k+1}D_X} + \sqrt{2\beta_2^{k+1}\delta_{k+1}}$ and $\phi(\bar{x}^{k+1}) - g(\bar{y}^{k+1}) \leq \beta_1^{k+1}D_X + \delta_{k+1}$.
Moreover, $\delta_{k+1}\leq \delta_0 = \tilde{\varepsilon}$ due to the choice of $\delta_0$ and the update rule of $\delta_k$ at Step 2 of Algorithm
\ref{alg:A1}. By combining these inequalities and \eqref{eq:beta_bounds} and then using the definition of $C_f$ and $\tau_0 = 0.5(\sqrt{5}-1)$ we obtain
\eqref{eq:duality_gap} and \eqref{eq:feasible_gap}.
The last conclusion is a direct consequence of \eqref{eq:duality_gap} and \eqref{eq:feasible_gap}.
\Eproof
\end{proof}
%+ End of the proof.

The conclusions of Theorem \ref{th:alg_scheme} show that the initial accuracy of solving the primal subproblems \eqref{eq:smoothed_gi} needs to be chosen
as $O(1/k)$. Then, we have $\abs{\phi(\bar{x}^k) - g(\bar{y}^k)} = O(1/k^{\alpha^{*}})$ and $\mathcal{F}(\bar{x}^k) = O(1/k)$. Thus, if we choose the ratio
$\alpha^{*}$ such that $\alpha^{*} \to 1^{-}$ then we obtain an asymptotic convergence rate $O(1/k)$ for Algorithm \ref{alg:A1}.
We note that the accuracy of solving \eqref{eq:smoothed_gi} has to be updated at each iteration $k$ in Algorithm \ref{alg:A1}. The new
value is computed by $\bar{\varepsilon}^{k} =  \tau_k\delta_k/Q_k$ at Step 2, which is the same $O(1/k^2)$ order.

Now, we consider a particular case, where we can  get an $O(1/\varepsilon)$ worst-case complexity ($\varepsilon$ is a desired accuracy).

%+ Corollary 4.1.
\begin{corollary}\label{co:worst_case_complexity}
Suppose that the smoothness parameter  $\beta_1^k$ in Algorithm \ref{alg:A1} is fixed at $\beta_1^k = \beta^0 = \sqrt{L_{\mathrm{A}}}\varepsilon_f$ for all
$k\geq 0$. Suppose further that the accuracy level $\tilde{\varepsilon}$ in Algorithm \ref{alg:A1} is chosen as $O(\varepsilon)$ and that
the sequence $\{\tau_k\}$ is updated by $\tau_{k+1} := 0.5\tau_k\left(\sqrt{\tau^2_k + 4} - \tau_k\right)$ starting from $\tau_0 := 0.5(\sqrt{5}-1)$. Then,
after $\bar{k} = \lfloor 2/\varepsilon_f\rfloor + 1$ iterations, one has
\begin{equation}\label{eq:worst_case_complexity}
\mathcal{F}(\bar{x}^{\bar{k}}) \leq C_f^0\varepsilon_f ~~\textrm{and}~~ \abs{\phi(\bar{x}^{\bar{k}}) - g(\bar{y}^{\bar{k}})} \leq C_d^0\varepsilon_f,
\end{equation}
where $C_f^0 := \sqrt{L_{\mathrm{A}}}(2R_{Y^{*}} + \sqrt{2D_X})$ and $C_d^0 := \sqrt{L_{\mathrm{A}}}\max\set{D_X, 2R_{Y^{*}}+\sqrt{2D_X}}$.
\end{corollary}

%+ Proof of Corollary 4.1.
\begin{proof}
If we assume that $\beta_1^k$ is fixed in Algorithm \ref{alg:A2} then, by the new update rule of  $\{\tau_k\}$ we have $\beta_2^{k+1}\beta_1^0 =
L_{\mathrm{A}}\tau_k^2 \leq \frac{4L_{\mathrm{A}}\tau_0^2}{(\tau_0k + 2)^2}$ due to \eqref{eq:tau_bounds} and \eqref{eq:beta_bounds} with
$\alpha^{*} = 0$. Since $\beta_1^0 = \sqrt{L_{\mathrm{A}}}\varepsilon_f$, if we choose $\bar{k} := \lfloor 2/\varepsilon_f\rfloor + 1$ then
$\frac{2\tau_0}{\tau_0(\bar{k}-1) +
2} \leq \varepsilon_f$. Furthermore, by Lemma \ref{le:excessive_gap} we have $\mathcal{F}(\bar{x}^{\bar{k}}) \leq 2\beta_2^{\bar{k}}R_{Y^{*}} +
\sqrt{2\beta_1^0\beta_2^{\bar{k}}D_X} \leq \sqrt{L_{\mathrm{A}}}(2R_{Y^{*}} + \sqrt{2D_X})\varepsilon_f$ and $-R_{Y^{*}}\mathcal{F}(\bar{x}^{\bar{k}}) \leq
\phi(\bar{x}^{\bar{k}}) - g(\bar{y}^{\bar{k}}) \leq \beta_1^0D_X = \sqrt{L_{\mathrm{A}}}D_X\varepsilon_f$. By combining these estimates, we obtain the
conclusion \eqref{eq:worst_case_complexity}.
\Eproof
\end{proof}
%+ End of the proof.

%+ Remark 2.
\begin{remark}{(\textit{Distributed implementation})}\label{re:dist_opt}
In Algorithm \ref{alg:A1}, only the parameter $\alpha_k$ requires centralized information. Instead of using $\alpha_k$, we can use its lower bound $\alpha^{*}$
to compute $\tau_k$ and $\beta_1^k$. In this case, we can modify Algorithm \ref{alg:A1} to obtain a distributed implementation.
The modification is at Steps 5, 6 and 7, where we can parallelize these steps by using the same formulas for the all subsystems to compute the parameters
$\beta_1^k$, $\beta_2^k$ and $\tau_k$.
We note that the points $\tilde{x}^{*}(\hat{y};\beta_1)$ and $\bar{x}^{k+1}$ in the scheme \eqref{eq:alg_scheme} can be computed in parallel, while
$y^{*}(\bar{x};\beta_2)$ and $\bar{y}^{+}$ can be computed distributively based on the structure of the coupling constraints of problem
\eqref{eq:primal_cvx_prob}.
\end{remark}

%%%%%%%%%%%%%%%%%%%%%%%%%%%%%%%%%%%%%%%%%%%%%%%%%%%%%%%%%%%%%%%%%%%%%%%%%%%%%%%%%%%%%%%%%%%%%%%%%%%%
%+ 5. An algorithm variant with switching strategy.
%%%%%%%%%%%%%%%%%%%%%%%%%%%%%%%%%%%%%%%%%%%%%%%%%%%%%%%%%%%%%%%%%%%%%%%%%%%%%%%%%%%%%%%%%%%%%%%%%%%%
\section{Inexact decomposition algorithm with switching primal-dual steps}\label{sec:inexact_switching_alg}
Since the  ratio $\alpha^{*}:= \frac{p_X^{*}}{D_X}$ defined in \eqref{eq:alpha_quatity} may be small, Algorithm \ref{alg:A1} only provides a suboptimal
approximation $(\bar{x}^k, \bar{y}^k)$ to the optimal solution $(x^{*}, y^{*})$ such that $\abs{\phi(\bar{x}^k) - g(\bar{y}^k)} \leq \beta^0D_X +
\tilde{\varepsilon}$ in the worst-case. For example, if we choose the prox-function $p_X(x) := \frac{1}{2}\sum_{i=1}^M\norm{x_i - x^c_i}^2 + \alpha^{*}$, where
$\alpha_{*} \in (0, 0.5)$, then worst-case complexity of Algorithm \ref{alg:A1} is lower than subgradient methods, see \eqref{eq:duality_gap} of Theorem
\ref{th:convergence}. Algorithm \ref{alg:A1} leads to a poor performance.

In this section, we propose to combine the scheme $\mathcal{S}^d$ defined by \eqref{eq:alg_scheme} in this paper and an
inexact decomposition scheme with two primal steps and one dual step to ensure that the parameter $\beta_1$ always decreases to zero.
Apart from the inexactness, this variant allows one to update simultaneously both smoothness parameters at each iteration.

%+ ******************************************************************
%+ 5.1. The inexact primal scheme.
%+ ******************************************************************
\vskip0.1cm
\noindent\textbf{5.1. The inexact main iteration with two primal steps.}
Let us consider the  approximate function $f(x;\beta_2) = \phi(x) + \psi(x;\beta_2)$ defined by \eqref{eq:smoothed_phi}. We recall that $\phi$ is only assumed
to be convex and possibly nonsmooth, while $\psi(\cdot;\beta_2)$ is convex and Lipschitz continuously differentiable. We define
\begin{equation}\label{eq:proxgrad_approx}
q_i(x_i; \hat{x}, \beta_2) \!:=\! \phi_i(x_i) \!+\! M^{-1}{\!\!}\psi(\hat{x};\beta_2) \!+\! \nabla_{x_i}{\psi}(\hat{x};\beta_2)^T(x_i \!-\! \hat{x}_i) \!+\!
\frac{L_i^{\psi}(\beta_2)}{2}\norm{x_i \!-\! \hat{x}_i}^2,
\end{equation}
and the mapping
\begin{equation}\label{eq:Pi_mapping}
P_i(\hat{x}, \beta_2) := \textrm{arg}\!\min_{x_i\in X_i} q_i(x_i; \hat{x}, \beta_2), ~~ i=1,\dots, M,
\end{equation}
where $L_i^{\psi}(\beta_2) := \frac{M\norm{A_i}^2}{\beta_2}$ is the   Lipschitz constant of $\nabla_{x_i}\psi(\cdot; \beta_2)$ defined in Lemma
\ref{le:primal_smoothed_estimates}.
Since $q_i(\cdot; \hat{x}, \beta_2)$ is strongly convex, $P_i(\hat{x}, \beta_2)$ is well-defined.

%% Remark 5.1.
\begin{remark}\label{re:bregman_dist}
Note that we can replace the quadratic term
$\frac{L_i^{\psi}(\beta_2)}{2}\norm{x_i \!-\! \hat{x}_i}^2$ in
\eqref{eq:proxgrad_approx} by any Bregman distance as done in
\cite{Nesterov2005a}. However, the convergence analysis based on
this type of prox-functions  is more complicated than the one given
in this paper.
\end{remark}

Suppose that we can only solve the minimization problem \eqref{eq:Pi_mapping} up to a given accuracy $\varepsilon_i \geq 0$ to obtain an approximate solution
$\widetilde{P}_i(\cdot, \beta_2)$ in the sense of Definition \eqref{de:inexactness}. More precisely, $\widetilde{P}_i(\hat{x},\beta_2)\in X_i$ and
\begin{eqnarray}\label{eq:Pi_approx}
0\leq q_i(\widetilde{P}_i(\hat{x},\beta_2); \hat{x}, \beta_2) - q_i(P_i(\hat{x}, \beta_2); \hat{x}, \beta_2) \leq \frac{1}{2}L_i^{\psi}(\beta_2)\varepsilon_i^2,
~~ i=1,\dots, M.
\end{eqnarray}
We denote $P := (P_1,\dots, P_M)$ and $\widetilde{P} := (\widetilde{P}_1, \dots, \widetilde{P}_M)$.
In particular, if $\phi_i$ is differentiable and  its gradient is Lipschitz continuous with a Lipschitz constant $L^{\phi_i} > 0$ for some $i\in \{1,2,\dots,
M\}$ then one can replace the
approximate mapping $\widetilde{P}_i$ by the following one:
\begin{eqnarray*}
\widetilde{G}_i(\hat{x}, \beta_2) {\!}:\approx{\!} \textrm{arg}\!\!\min_{x_i\in X_i}{\!\!\!}\left\{\!\!
\left[\nabla\phi_i(\hat{x}_i) \!+\! \nabla_{x_i}\psi(\hat{x};\beta_2)\right]^T\!\!\!\!(x_i \!-\! \hat{x}_i) \!+\! \frac{\hat{L}_i(\beta_2)}{2}\norm{x_i -
\hat{x}_i}^2\!\!\right\},
\end{eqnarray*}
where $\hat{L}_i(\beta_2) := L^{\phi_i} + L^{\psi}_i(\beta_2)$, in the sense of Definition \ref{de:inexactness}. Note that the minimization problem defined in
$\widetilde{G}_i$ is a quadratic program with convex constraints.

Now, we can present the decomposition scheme with two primal steps in the case of inexactness as follows.
Suppose that $(\bar{x}, \bar{y})\in X\times\mathbb{R}^m$ satisfies \eqref{eq:excessive_gap_inexact} w.r.t. $\beta_1$, $\beta_2$ and $\delta$. We update
$(\bar{x}^{+}, \bar{y}^{+})\in X\times\mathbb{R}^m$ as
\begin{equation}\label{eq:p_alg_scheme}
(\bar{x}^{+}, \bar{y}^{+}) := \mathcal{S}^p(\bar{x}, \bar{y}, \beta_1, \beta_2^{+}, \tau) ~\Leftrightarrow~
\left\{
\begin{array}{ll}
\hat{x} &:= (1-\tau)\bar{x} + \tau \tilde{x}^{*}(\bar{y}; \beta_1)\\
\bar{y}^{+} &:= (1-\tau)\bar{y} + \tau y^{*}(\hat{x}; \beta_2^{+})\\
\bar{x}^{+} &:= \widetilde{P}(\hat{x}, \beta_2^{+}),
\end{array}\right.
\end{equation}
where the step size $\tau \in (0, 1)$  will be appropriately updated and
\begin{enumerate}
\item the parameters $\beta_1$ and $\beta_2$ are updated by $\beta_1^{+} := (1-\tau)\beta_1$ and $\beta_2^{+} := (1-\tau)\beta_2$;
\item $\tilde{x}^{*}(\bar{y}; \beta_1)$ is computed by \eqref{eq:ap_smoothed_gi_sol};
\item $\widetilde{P}(\cdot, \beta_2^{+})$ is an approximation of $P(\cdot,\beta_2^{+})$ defined in \eqref{eq:Pi_mapping} and \eqref{eq:Pi_approx}.
\end{enumerate}
The following theorem states that the new point $(\bar{x}^{+}, \bar{y}^{+})$ updated by $\mathcal{S}^p$ maintains the $\delta_{+}$-excessive gap condition
\eqref{eq:excessive_gap_inexact}. The proof of this theorem is postponed to Appendix \ref{sec:proofs}.

%+ Theorem 5.1.
\begin{theorem}\label{th:alg_scheme2}
Suppose that Assumptions \aref{as:A1} and \aref{as:A2} are satisfied.
Let  $(\bar{x}, \bar{y})$ be a point in $X\times\mathbb{R}^m$ and satisfy the $\delta$-excessive gap condition
\eqref{eq:excessive_gap_inexact} w.r.t. two values $\beta_1$ and $\beta_2$. Then if the parameter $\tau$ is chosen such that $\tau\in (0, 1)$ and
\begin{equation}\label{eq:alg_scheme_cond_p}
\beta_1\beta_2 \geq \left(\frac{\tau}{1-\tau}\right)^2L_{\mathrm{A}},
\end{equation}
then the new points $(\bar{x}^{+}, \bar{y}^{+})$ updated by \eqref{eq:p_alg_scheme} maintains the $\delta_{+}$-excessive gap condition
\eqref{eq:excessive_gap_inexact} w.r.t. two new values $\beta_1^{+}$ and $\beta^{+}_2$,
where $\delta_{+} := (1-\tau)\delta + 2\beta_1(1-\tau)D_{\sigma}\varepsilon_{[\sigma]} + \frac{1}{2}\sum_{i=1}^ML_i^{\psi}(\beta_2^{+})\varepsilon_i^2$,
and $\varepsilon_{[\sigma]}$ and $D_{\sigma}$ are defined in \eqref{eq:accuracy_constants}.
\end{theorem}

Finally, we note that the step size $\tau$ is updated by  $\tau_{k+1} := \tau_k/(\tau_k + 1)$ for $k\geq 0$ starting from $\tau_0 := 0.5$ in the
scheme \eqref{eq:p_alg_scheme}, see \cite{TranDinh2012a} for more details.

%+ ******************************************************************
%+ 5.2. The algorithm and its convergence.
%+ ******************************************************************
\vskip0.1cm
\noindent\textbf{5.2. The algorithm and its convergence.}
First, we provide an update rule for $\delta$ in Definition \ref{de:exessive_gap_cond}.
With $\varepsilon_{[\sigma]}$ and $D_{\sigma}$ defined in \eqref{eq:accuracy_constants}, let us consider the function
\begin{equation}\label{eq:xi_func}
\xi(\tau, \beta_1, \beta_2, \varepsilon) := 2\beta_1(1-\tau)D_{\sigma}\varepsilon_{[\sigma]} +
\frac{1}{2}\sum_{i=1}^ML_i^{\psi}(\beta_2^{+})\varepsilon_i^2,
\end{equation}
and a sequence  $\{\delta_k\}$ generated by $\delta_{k+1} := (1-\tau_k)\delta_k + \xi(\tau_k, \beta_1^k, \beta_2^k, \varepsilon^k)$, where
$\delta_0$ is given and $\varepsilon^k$ is chosen appropriately.
The aim is to choose $\bar{\varepsilon}_k$ such that $0\leq\varepsilon_i^k\leq\bar{\varepsilon}_k$ and $\{\delta_k\}$ is nonincreasing.
By letting
\begin{equation}\label{eq:R_k_const}
R_k := 2(1-\tau_k)\beta_1^kD_{\sigma}\left(\sum_{i=0}^M\sigma_i\right)^{1/2} +  \frac{M}{2(1-\tau_k)\beta_2^k}\sum_{i=1}^M\norm{A_i}^2,
\end{equation}
Then, if we choose $\bar{\varepsilon}^k\geq 0$ such that $\bar{\varepsilon}^k \leq \frac{\tau_k\delta_k}{R_k}$ then we have $\delta_{k+1} \leq \delta_k$.

By combining both schemes \eqref{eq:alg_scheme} and \eqref{eq:p_alg_scheme}, we obtain a new variant of Algorithm \ref{alg:A1} with a switching strategy as
described as follows.

%+ Algorithm A.2.
%%%%%%%%%%%%%%%%%%%%%%%%%%%%%%%%%%%%%%%%%%%%%%%%%%%%%%%%%%%%%%%%%%%%%%%%%
\begin{algorithm}{\textit{Inexact decomposition algorithm with switching primal-dual steps}}\label{alg:A2}
\newline
\noindent{\!\!\!\textbf{Initialization:}}
Perform as in  Algorithm \ref{alg:A1} with $\tau_0 := 0.5$.
\newline
\noindent\textbf{Iteration: }For $k=0,1,2, \dots, k_{\max}$ perform the following steps:
\begin{itemize}
\item[]\textit{Step 1}: If a given stopping criterion is satisfied then terminate.
\item[]\textit{Step 2}: If $k$ is \textit{even} then perform the scheme with two primal steps:
\begin{itemize}
\item[]2.1. Compute $R_k$ by \eqref{eq:R_k_const}. Set $\bar{\varepsilon}^k := \tau_k\delta_k/R_k$ and update $\delta_{k+1} := (1-\tau_k)\delta_k
+ R_k\bar{\varepsilon}^k$.
\item[]2.2. Update  $\beta_2^{k+1} := (1-\tau_k)\beta_2^k$.
\item[]2.3. Compute $(\bar{x}^{k+1}, \bar{y}^{k+1}) := \mathcal{S}^p(\bar{x}^k, \bar{y}^k, \beta_1^{k}, \beta_2^{k+1}, \tau_k)$ up to the accuracy
$\bar{\varepsilon}^k$.
\item[]2.4. Update  $\beta_1^{k+1} := (1-\tau_k)\beta_1^k$.
\item[]2.5. Update the step-size parameter $\tau_k$ as $\tau_{k+1} := \frac{\tau_k}{\tau_k + 1}$.
\end{itemize}
\item[]\textit{Step 3}: Otherwise, (i.e. $k$ is \textit{odd}) perform the scheme with two dual steps:
\begin{itemize}
\item[]3.1. Compute $Q_k$ by \eqref{eq:accuracy_level_choice}. Set $\bar{\varepsilon}^k := \tau_k\delta_k/Q_k$ and update $\delta_{k+1} := (1-\tau_k)\delta_k
+ Q_k\bar{\varepsilon}^k$.
\item[]3.2. Compute $(\bar{x}^{k+1}, \bar{y}^{k+1}) := \mathcal{S}^d(\bar{x}^k, \bar{y}^k, \beta_1^k, \beta_2^k, \tau_k)$ up to the accuracy
$\bar{\varepsilon}^k$.
\item[]3.3. Compute $\tilde{\alpha}_k := \frac{p_X(\tilde{x}^{*}(\hat{y}^k;\beta_1^k))}{D_X}$, where $\hat{y}^k := (1-\tau_k)\bar{y}^k +
\tau_k(\beta_2^k)^{-1}(A\bar{x}^k-b)$.
\item[]3.4. Update  $\beta_1^{k+1} := (1 - \tilde{\alpha}_k\tau_k)\beta_1^k$ and $\beta_2^{k+1} := (1-\tau_k)\beta_2^k$.
\item[]3.5. Update $\tau_k$ as $\tau_{k+1} := \frac{\tau_k}{2}\left\{\left[(1-\tilde{\alpha}_k\tau_k)^2\tau_k^2 + 4(1-\tilde{\alpha}_k\tau_k)\right]^{1/2} -
(1-\tilde{\alpha}_k\tau_k)\tau_k\right\}$.
\end{itemize}
\end{itemize}
\end{algorithm}
%%%%%%%%%%%%%%%%%%%%%%%%%%%%%%%%%%%%%%%%%%%%%%%%%%%%%%%%%%%%%%%%%%%%%%%%%

Note that the first line and third line of the scheme $\mathcal{S}^p$ can be parallelized. They require one to solve $M$ convex subproblems of
the form \eqref{eq:smoothed_gi} and \eqref{eq:Pi_approx}, respectively \textit{in parallel}.
If the function $\phi_i$ is differentiable and its gradient is Lipschitz continuous for some $i\in\{1,\dots, M\}$, then we can use the approximate
gradient mapping $\widetilde{G}_i$ instead of $\widetilde{P}_i$ and the corresponding minimization subproblem in the third line reduces to  a
quadratic program with convex constraints. The stopping criterion at Step 1 will be given in Section \ref{sec:num_results}.

Similar to the proof of Lemma \ref{le:choice_of_tau} we can show that the sequence $\{\tau_k\}_{k\geq 0}$ generated by Step 2.5 or Step 3.5 of Algorithm
\ref{alg:A2} satisfies  estimates \eqref{eq:tau_bounds}. Consequently, the estimate for $\beta_2^k$ in \eqref{eq:beta_bounds} is still valid, while the
parameter $\beta_1^k$ satisfies $\beta_1^{k+1} \leq \frac{\beta_1^0}{(\tau_0k + 1)^{(1+\alpha^{*})/2}}$.

Finally, we summarize the convergence results of Algorithm \ref{alg:A2} in the following theorem.

%+ Theorem 5.2.
\begin{theorem}\label{th:convergence2}
Suppose that Assumptions \aref{as:A1} and \aref{as:A2}  are satisfied.
Let $\{(\bar{x}^k,\bar{y}^k)\}$ be a sequence generated by Algorithm \ref{alg:A2} after $\bar{k}$ iterations. If the accuracy level
$\tilde{\varepsilon}$ in Algorithm \ref{alg:A2} is chosen  such that $0 \leq \tilde{\varepsilon} \leq \frac{c_0}{0.5\bar{k} + 1}$ for some positive constant
$c_0$, then the following primal-dual gap holds
\begin{equation}\label{eq:duality_gap2}
-R_{Y^{*}}\mathcal{F}(\bar{x}^{\bar{k} + 1}) \leq \phi(\bar{x}^{\bar{k}+1}) - g(\bar{y}^{\bar{k}+1}) \leq \frac{\beta^0D_X + c_0}{(0.5\bar{k} +
1)^{(1+\alpha^{*})/2}},
\end{equation}
and the feasibility gap satisfies
\begin{equation}\label{eq:feasible_gap2}
\mathcal{F}(\bar{x}^{\bar{k}+1}) = \norm{A\bar{x}^{\bar{k}+1} - b} \leq \frac{C_f}{0.25(1+\alpha^{*})k + 1},
\end{equation}
where $C_f :=  \frac{L_{\mathrm{A}}R_{Y^{*}}}{\beta^0} + 0.5\sqrt{2L_{\mathrm{A}}(D_X + c_0/\beta_0)}$ and $R_{Y^{*}}$ is
defined as in
\eqref{eq:simplified_feasibility_gap}.

Consequently, the sequence  $\{(\bar{x}^k, \bar{y}^k)\}_{k\geq 0}$ generated by Algorithm \ref{alg:A2} converges to a solution $(\bar{x}^{*}, \bar{y}^{*})$ of
the primal and dual problems \eqref{eq:primal_cvx_prob}-\eqref{eq:dual_cvx_prob} as $k\to\infty$ and $\tilde{\varepsilon}\to 0^{+}$.
\end{theorem}

The proof of this theorem is  similar to Theorem \ref{th:convergence} and thus we omit the details here.
We can see from the right  hand side of \eqref{eq:duality_gap2} in Theorem \ref{th:convergence2} that this term is better than the one in Theorem
\ref{th:convergence}. Consequently, the worst case complexity of Algorithm \ref{alg:A2} is better than the one of Algorithm \ref{alg:A1}. However, as a
compensation, at each even iteration, the scheme $\mathcal{S}^p$ is performed. It requires an additional cost to compute $\bar{x}^{+}$ at the third line of
$\mathcal{S}^p$. As an exception, if the primal subproblem \eqref{eq:smoothed_gi} can be solved in a \textit{closed form} then the cost-per-iteration of
Algorithm \ref{alg:A2} is almost the same as in Algorithm \ref{alg:A1}.

%+ Remark 5.1.
\begin{remark}\label{re:primal_update_variant}
Note that we can only use the  inexact decomposition scheme with two primal steps  $\mathcal{S}^p$ in \eqref{eq:p_alg_scheme} to build an inexact
variant of \cite[Algorithm 1]{TranDinh2012a}.
Moreover, since the role of the schemes  $\mathcal{S}^p$ and $\mathcal{S}^d$ is symmetric, we can switch them in Algorithm \ref{alg:A2}.
\end{remark}

%%%%%%%%%%%%%%%%%%%%%%%%%%%%%%%%%%%%%%%%%%%%%%%%%%%%%%%%%%%%%%%%%%%%%%%%%%%%%%%%%%%%%%%%%%%%%%%%%%%%%%%%%%
%+ 7. Numerical examples.
%%%%%%%%%%%%%%%%%%%%%%%%%%%%%%%%%%%%%%%%%%%%%%%%%%%%%%%%%%%%%%%%%%%%%%%%%%%%%%%%%%%%%%%%%%%%%%%%%%%%%%%%%%
\section{Numerical tests}\label{sec:num_results}
In this section we  compare Algorithms \ref{alg:A1} and \ref{alg:A2} derived in this paper with the two algorithms developed in
\cite[Algorithms 1 and 2]{TranDinh2012a} which we named \texttt{2pDecompAlg} and \texttt{pdDecompAlg}, the proximal center-based decomposition algorithm in
\cite{Necoara2008}, an exact variant of the proximal based decomposition algorithm in \cite{Chen1994} and three parallel variants of the alternating
direction method of multipliers (with three different strategies to update the penalty parameter). We note that  these variants are the
modifications of the algorithm in \cite{Lenoir2007}, and they can be applied to solve problem \eqref{eq:primal_cvx_prob} with more than
two objective components (i.e. $M > 2$). We named these algorithms by \texttt{PCBDM}, \texttt{EPBDM}, \texttt{ADMM-v1},
\texttt{ADMM-v2} and \texttt{ADMM-v3}, respectively. 
For more simulations and comparisons we refer  to the extended technical report \cite{TranDinh2013a}.

The algorithms have been implemented in C++ running on a $16$ cores Intel \textregistered Xeon $2.7$GHz workstation with $12$ GB of RAM. In order to solve the
general convex programming subproblems, we either used a commercial software called \texttt{Cplex} or an open-source software package \texttt{IpOpt}
\cite{Waechter2006}. All the algorithms have been parallelized by using \texttt{OpenMP}.

In the four numerical examples below, since the feasible set $X_i$ has no specific structure, we chose the quadratic prox-function $p_{X_i}(x_i) :=
\frac{1}{2}\norm{x_i - x^c_i}^2 + r_i$ in the four first algorithms, i.e. Algorithms \ref{alg:A1} and \ref{alg:A2}, \texttt{2pDecompAlg} and
\texttt{pdDecompAlg}, where $x_i^c\in\mathbb{R}^{n_i}$ and $r_i = 0.75D_{X_i}$ are given, for $i=1,\dots, M$, as mentioned in Remark \ref{re:quad_prox}.
With this choice we can solve the primal subproblem \eqref{eq:smoothed_gi} in the first example by using \texttt{Cplex}.

We terminated these algorithms if
\begin{equation}\label{eq:term_cond1}
\texttt{rpfgap} := \norm{A\bar{x}^k - b}/\max\{\norm{b}, 1\} \leq 10^{-3},
\end{equation}
and either the approximate primal-dual gap satisfied
\begin{equation*}
\abs{f(\bar{x}^k;\beta_2^k) - \tilde{g}(\bar{y}^k;\beta_1^k)} \leq 10^{-3}\max\left\{1.0, \abs{\tilde{g}(\bar{y}^k;\beta_1^k)},
\abs{f(\bar{x}^k;\beta_2^k)}\right\},
\end{equation*}
or the value of the objective function did not significantly change in $5$ successive iterations, i.e.:
\begin{equation}\label{eq:term_cond2}
\abs{\phi(\bar{x}^{k}) - \phi(\bar{x}^{k-j})}/\max\{1.0, \abs{\phi(\bar{x}^k)}\}\leq 10^{-3} ~\textrm{for} ~j=1,\dots, 5.
\end{equation}
Here $\tilde{g}(\bar{y}^k;\beta_1^k)$ is the approximate value of $g(\bar{y}^k;\beta_1^k)$ evaluated at
$\tilde{x}^{*}(\bar{y}^k;\beta_1^k)$.

In \texttt{ADMM-v1} and \texttt{ADMM-v2} we used the update formula in \cite[formula (21)]{Boyd2011} to update the penalty parameter $\rho_k$ starting from
$\rho_0 := 1$ and $\rho_0 := 1000$, respectively. In \texttt{ADMM-v3} this penalty parameter was fixed at $\rho_k := 1000$ for all iterations.
In  \texttt{PCBDM}, we chose the same  prox-function as in our algorithms and the parameter $\beta_1$ in the subproblems was fixed at $\beta_1 :=
\frac{\varepsilon_{\mathrm{p}}\max\left\{1.0, \abs{\phi(\bar{x}^0)}\right\}}{D_X}$.
We terminated all the remaining algorithms if the both conditions \eqref{eq:term_cond1} and \eqref{eq:term_cond2} were satisfied.
The maximum number of iterations $\texttt{maxiter}$ was set to $5000$ in all algorithms.
We warm-started the \texttt{Cplex} and \texttt{IpOpt} solvers at the iteration $k$ at the point given by the previous iteration $k-1$ for $k\geq 1$.
The accuracy levels $\bar{\varepsilon}_k$ in \texttt{Cplex} and \texttt{IpOpt} and $\delta_k$ were updated as in Algorithms \ref{alg:A1} and \ref{alg:A2}
starting from $\delta_0 = 10^{-3}$ and then set to $\max\{\bar{\varepsilon}_k, 10^{-10}\}$. In other algorithms, this accuracy level was fixed at
$\bar{\varepsilon}_k = 10^{-8}$.
We concluded that ``the algorithm is failed'' if either the maximum number of iterations \texttt{maxiter} was reached or the primal subproblems
\eqref{eq:smoothed_gi} could not be solved by \texttt{IpOpt} or \texttt{Cplex} due to numerical issues.

We benchmarked all algorithms with performance profiles \cite{Dolan2002}. Recall that a performance profile is built based on a set $\mathcal{S}$ of $n_s$
algorithms (solvers) and a collection $\mathcal{P}$ of $n_p$ problems. Suppose that we build a profile based on computational time. We denote by
$T_{p,s} := \textit{computational time required to solve problem $p$ by solver $s$}$.
We compare the performance of algorithm $s$ on problem $p$ with the best performance of any algorithm on this problem; that is we compute the performance ratio
$r_{p,s} := \frac{T_{p,s}}{\min\{T_{p,\hat{s}} ~|~ \hat{s}\in \mathcal{S}\}}$.
Now, let $\tilde{\rho}_s(\tilde{\tau}) := \frac{1}{n_p}\mathrm{size}\left\{p\in\mathcal{P} ~|~ r_{p,s}\leq \tilde{\tau}\right\}$ for
$\tilde{\tau} \in\mathbb{R}_{+}$. The function $\tilde{\rho}_s:\mathbb{R}\to [0,1]$ is the probability for solver $s$ that a performance ratio is within a
factor $\tilde{\tau}$ of the best possible ratio. We use the term ``performance profile'' for the distribution function $\tilde{\rho}_s$ of a performance
metric.
In the following numerical examples, we plotted the performance profiles in $\log_2$-scale, i.e. $\rho_s(\tau) :=
\frac{1}{n_p}\mathrm{size}\left\{p\in\mathcal{P} ~|~ \log_2(r_{p,s})\leq
\tau := \log_2\tilde{\tau}\right\}$.

%+ ******************************************************************
%+ 6.1. Basic pursuit problem.
%+ ******************************************************************
\vskip0.1cm
\noindent\textbf{6.1. Basic pursuit problem.}
The basic pursuit problem is one of the fundamental problems in signal processing and compressive sensing. 
Mathematically, this problem can be formulated as follows:
\begin{equation}\label{eq:basic_lasso}
\begin{cases}
\min &\norm{x}_1\\
\textrm{s.t.} & Ax = b,
\end{cases}
\end{equation}
where $A \in \mathbb{R}^{m\times n}$ and $b\in\mathbb{R}^m$ are given.
Since $\phi(x) = \norm{x}_1 = \sum_{i=1}^n\phi_i(x_i) = \sum_{i=1}^n\abs{x_i}$, the primal subproblem
\eqref{eq:smoothed_gi} formed from \eqref{eq:basic_lasso}
in the algorithms can be expressed as
\begin{equation*}
\min_{x_i\in\mathbb{R}}\left\{|x_i| + (A_i^Ty)x_i + \frac{\beta_1}{2}(x_i-x^c_i)^2\right\}.
\end{equation*}
This problem can be solved in a closed form without any subiteration. 
We implemented Algorithms \ref{alg:A1} and \ref{alg:A2} to solve this problem in order to compare the effect of the parameter $\alpha^{*}$ on the performance
of the algorithm. The data of this problem is generated as follows. Matrix $A$ is generated randomly such that it is orthogonal. Vector $b := Ax^0$, where
$x^0$ is a $k$-sparse random vector ($k = \lfloor0.05n\rfloor$). We tested  Algorithms \ref{alg:A1} and \ref{alg:A2} with $5$ problems and the results reported
by these algorithms are presented in Table
\ref{tb:basic_lasso} with $\alpha^{*} = 0.25$ and $\alpha^{*} = 0.75$.
\begin{table}[!ht]
\begin{center}
\begin{footnotesize}
\newcommand{\cell}[1]{{\!\!\!}#1{\!\!\!}}
\newcommand{\cellbf}[1]{{\!\!\!\!}\textbf{#1}{\!\!\!\!}}
\caption{Performance comparison of Algorithms 1 and 2 for solving \eqref{eq:basic_lasso} (This test was done on a MAC book Laptop (Intel 2.6GHz core i7, 
16GB Ram). The information in rows 3, 4, 6 and 7, and columns 2-6 is the number of iterations / the computational time in second}\label{tb:basic_lasso}
\begin{tabular}{|c|r|r|r|r|r|}\hline
% \multicolumn{1}{|c|}{\!\!\!\!~\!\!\!\!}  & \multicolumn{5}{|c|}{Algorithm performance and results} \\ \hline
\multicolumn{1}{|c|}{[$(m,n)$]} & (20,124) & (50,128) & (80,256) & (100, 680) & (100, 1054) \\ \hline
%+ Iteration
 \multicolumn{6}{|c|}{$\alpha^{*} = 0.25$} \\ \hline
  \texttt{Algorithm 1}  & \cell{8254/1.3858} &\cell{5090/0.8769} &\cell{10144/2.2876} &\cell{29773/7.8386} &\cell{35615/10.9315} \\
 \texttt{Algorithm 2}  & \cell{4836/0.9025} &\cell{4744/0.8808} &\cell{8060/1.9809} &\cell{13220/3.7619} &\cell{12348/3.7967} \\ \hline
 \multicolumn{6}{|c|}{$\alpha^{*} = 0.75$} \\ \hline
 \texttt{Algorithm 1}  & \cell{7115/1.1851} &\cell{6644/1.1632} &\cell{8749/1.9632} &\cell{14927/4.0424} &\cell{16128/4.9169} \\
 \texttt{Algorithm 2}  & \cell{15016/2.6689} &\cell{14284/2.6121} &\cell{17140/4.0286} &\cell{34048/9.6910} &\cell{36668/11.1801} \\ \hline
\end{tabular}
\end{footnotesize}
\end{center}
\end{table}
As we can see from this  table that Algorithm \ref{alg:A2}  performs better than Algorithm \ref{alg:A1} in terms of number of iterations as well as
computational time for the case $\alpha^{*} = 0.25$. In the case $\alpha^{*} = 0.75$, Algorithm \ref{alg:A1} performs better than Algorithm \ref{alg:A2}. 
This example claims the theoretical results.

%+ ******************************************************************
%+ 6.2. Nonsmooth separable convex optimization. 
%+ ******************************************************************
\vskip0.1cm
\noindent\textbf{6.2. Nonsmooth separable convex optimization. }
Let us consider the following simple nonsmooth convex optimization problem:
\begin{equation}\label{eq:example}
\left\{\begin{array}{cl}
\displaystyle\min_{x\in\mathbb{R}^n}& \phi(x) := \sum_{i=1}^ni\abs{x_i - x^a_i},\\
\mathrm{s.t.} ~& \sum_{i=1}^nx_i = b,~ x_i \in X_i, ~i=1,\dots, n,
 \end{array}\right.
\end{equation}
where $b, x_i^a\in\mathbb{R}$ are given $(i=1,\dots, n)$. 
Let us assume that $x_i\in X_i := [l_i, u_i]$ is a given interval in $\mathbb{R}$. 
Then, this problem can be formulated in the form of \ref{eq:primal_cvx_prob} with $M = n$. 
Since the Lagrange function $\mathcal{L}(x, y) = \sum_{i=1}^n\left[ i\abs{x_i - x^a_i} + y(x_i-b/n)\right]$ is nonsmooth, where $y\in\mathbb{R}$ is a Lagrange
multiplier, we choose $p_{X_i}(x_i) := \frac{1}{2}\norm{x_i - x_i^c}^2 + 0.75D_{X_i}$ such that the primal subproblem can be written as
\begin{equation}\label{eq:exam1_smoothed_gi}
g_i(y; \beta_1) \!:= \!\!\! \min_{x_i\in [l_i, u_i]}\!\!\!\!\Big\{ i\abs{x_i \!-\! x_i^a} \!+\! y\big(x_i \!-\! \frac{b}{n}\big) \!+\!
\frac{\beta_1}{2}\abs{x_i \!-\!
x_i^c}^2 \!+\! 0.75D_{X_i}\Big\},
\end{equation}
where $\beta_1 > 0$. Now, we assume that we can choose the interval $[l_i, u_i]$ sufficiently large such that the constraint $x_i\in
[l_i, u_i]$ is inactive. Then, the solution of problem \eqref{eq:exam1_smoothed_gi} can be computed explicitly as $x_i^{*}(y;\beta_1) := V_i(x_i^a, x_i^c, y,
\beta_1, i)$, where the \textit{soft-thresholding-type operator} $V_i$ is defined as follows:
\begin{equation}\label{eq:exam1_shrinkage_opt}
V_i(x_i^a, x_i^c, y, \beta_1, \gamma) \!:=\! \begin{cases}
x^c_i \!-\! \beta_1^{-1}(\gamma \!+\! y) & \mathrm{if}~ x^c_i \!-\! \beta_1^{-1}(\gamma \!+\! y) > x_i^a,\\
x^c_i + \beta_1^{-1}(\gamma - y) & \mathrm{if}~ x^c_i + \beta_1^{-1}(\gamma - y) < x_i^a,\\
x^a_i                            & \mathrm{if}~ y \!+\! \beta_1(x_i^a \!-\! x_i^c) \!\in\! [-\gamma, \gamma].
\end{cases}
\end{equation}
In this example, we tested  five algorithms: Algorithm \ref{alg:A1}, Algorithm \ref{alg:A2}, \cite[Algorithm 1]{TranDinh2012a}, \cite[Algorithm
2]{TranDinh2012a} and \texttt{PCBDM} for $10$ problems with the size varying from $n = 5$ to $n = 100,000$. Note that if we reformulate
\eqref{eq:example} as a linear programming problem (LP) by introducing slack variables, then the resulting LP problem has $2n$
variables and $2n+1$ inequality constraints.

The data of these tests were created as follows. The value $c$ was set to $b = 2n$, $x^a := (x^a_1,\dots, x^a_n)^T$, where $x_a^i := i
- n/2$. The maximum number of iterations  \texttt{maxiter} was increased to $10,000$ instead of $5,000$. The performance  of the
five algorithms is reported in Table \ref{tb:exam1_results}.
\begin{table}[!ht]
\begin{center}
\begin{footnotesize}
\newcommand{\cell}[1]{{\!\!\!}#1{\!\!\!}}
\newcommand{\cellbf}[1]{{\!\!\!\!}\textbf{#1}{\!\!\!\!}}
\caption{Performance comparison of five algorithms for solving \eqref{eq:example}}\label{tb:exam1_results}
\begin{tabular}{|c|c|r|r|r|r|r|r|r|r|r|r|}\hline
\multicolumn{2}{|c|}{\!\!\!\!~\!\!\!\!}  & \multicolumn{10}{|c|}{Algorithm performance and results} \\ \hline
\multicolumn{2}{|c|}{\!\! Size [$n$] \!\!\!} & 5 & 10 & 50 & 100 & 500 & 1,000 &{\!\!\!} 5,000 {\!\!} & {\!\!\!}10,000{\!\!} &{\!\!\!} 50,000{\!\!} &
\!\!\!{100,000}\!\! \\ \hline
%+ Iteration
&  \!\!\!\!\texttt{2pDecompAlg}\!\!\!\!
& \cell{226} &\cell{184} &\cell{704} &\cell{843} &\cell{1211} &\cell{1277} &\cellbf{1371} &\cellbf{1387} &\cell{1408} &\cell{1409}\\
& \!\!\!\!\texttt{Algorithm 1}\!\!\!\!
& \cell{1216} &\cell{925} &\cellbf{377} &\cellbf{552} &\cellbf{1092} &\cellbf{1209} &\cell{1385} &\cell{1422} &\cell{1374} &\cellbf{1352} \\
\!\!\!\!\texttt{iters}\!\!\!\!
&\!\!\!\!\texttt{Algorithm 2}\!\!\!\!
& \cell{452} &\cell{334} &\cell{544} &\cell{794} &\cell{1142} &\cell{1228} &\cell{1415} &\cell{1433} &\cellbf{1358} &\cell{1368} \\
&  \!\!\!\!\texttt{pdDecompAlg}\!\!\!\!
& \cell{612} &\cell{458} &\cell{830} &\cell{887} &\cell{1253} &\cell{1341} &\cell{1451} &\cell{1428} &\cell{1487} &\cell{1446} \\
& \!\!\!\!\cell{\texttt{PCBDM}}\!\!\!\!
& \cellbf{62} &\cellbf{123} &\cell{507} &\cell{1036} &\cell{  3767} &\cell{  3693} &\cell{  6119} &\cell{  5816} &\cell{  3099} &\cell{  3285} \\
\hline
% CPU Time.
&  \!\!\!\!\texttt{2pDecompAlg}\!\!\!\!
& \cell{0.0143} &\cell{0.0105} &\cell{0.0339} &\cell{0.0495} &\cell{0.0809} &\cell{0.1078} &\cell{0.2969} &\cell{0.5943} &\cell{2.5055} &\cell{4.9713} \\
& \!\!\!\!\texttt{Algorithm 1}\!\!\!\!
& \cell{0.0592} &\cell{0.0418} &\cellbf{0.0170} &\cellbf{0.0266} &\cellbf{0.0596} &\cellbf{0.0827} &\cellbf{0.2477} &\cell{0.4544} &\cell{2.1970}
&\cell{4.3869} \\
\!\!\texttt{time}\!\!\!
& \!\!\!\!\texttt{Algorithm 2}\!\!\!\!
& \cell{0.0244} &\cell{0.0166} &\cell{0.0222} &\cell{0.0406} &\cell{0.0737} &\cell{0.0909} &\cell{0.3522} &\cell{0.4646} &\cellbf{2.0875} &\cellbf{4.2659} \\
&  \!\!\!\!\texttt{pdDecompAlg}\!\!\!\!
& \cell{0.0316} &\cell{0.0199} &\cell{0.0351} &\cell{0.0450} &\cell{0.0716} &\cell{0.0979} &\cell{0.3013} &\cellbf{0.4416} &\cell{2.2879} &\cell{4.3119} \\
& \cell{\texttt{PCBDM}}
& \cellbf{0.0027} &\cellbf{0.0036} &\cell{0.0218} &\cell{12.1021} &\cell{0.2116} &\cell{0.2232} &\cell{1.1448} &\cell{1.3084} &\cell{3.0277} &\cell{6.3322} \\
\hline
\end{tabular}
\end{footnotesize}
\end{center}
\end{table}
Here, \texttt{iters} is the number of  iterations and \texttt{time}
is the CPU time in seconds.

As we can see from Table \ref{tb:exam1_results},  Algorithm \ref{alg:A1} is the
best in terms of number of iterations and computational time.
Algorithm \ref{alg:A2} works better than \texttt{pdDecompAlg}. The first four algorithms have consistently outperformed \texttt{PCBDM} in terms of
number of iterations as well as computational time in this example.

%+ ******************************************************************
%+ 6.3. Nonlinear smooth resource allocation problems.
%+ ******************************************************************
\vskip0.1cm
\noindent\textbf{6.3. Separable convex quadratic programming.}
Let us consider a separable convex quadratic program of the form:
\begin{equation}\label{eq:sqp_prob}
\left\{\begin{array}{cl}
\displaystyle\min_{x\in\mathbb{R}^n} &\set{\phi(x) := \sum_{i=1}^M\frac{1}{2}x_i^TQ_ix_i + q_i^Tx_i},\\
\textrm{s.t.} &\sum_{i=1}^MA_ix_i = b,\\
&x_i \geq 0, ~ i=0,\dots, M.
\end{array}\right.
\end{equation}
Here $Q_i\in\mathbb{R}^{n_i\times n_i}$ is a symmetric positive semidefinite matrix, $q_i\in\mathbb{R}^{n_i}$, $A_i\in\mathbb{R}^{m\times n_x}$ for $i=1,\dots,
M$ and $b\in\mathbb{R}^m$.
In this example, we compared the above algorithms by building their performance profiles in terms of number of iterations and the total computational time.

\vskip0.1cm
\noindent\textit{Problem generation. } The input data of the test was generated as follows.
Matrix $Q_i := R_iR^T_i$, where $R_i$ is an $n_i\times
r_i$ random matrix in $[l_Q,  u_Q]$ with $r_i := \lfloor n_i/2\rfloor$.
Matrix $A_i$ was generated randomly in $[l_A, u_A]$. Vector $q_i := -Q_ix^0_i$, where
$x^0_i$ is a given feasible point in $(0, r_{x_0})$ and vector $b := \sum_{i=1}^MA_ix^0_i$.
The density of both matrices $A_i$ and $R_i$ is $\gamma_{A}$.
Note that the problems generated as above are always feasible. Moreover, they are not strongly convex.
The tested collection consisted of $n_p = 50$ problems with different sizes and the sizes were generated randomly as follows:
\begin{itemize}
\item\textit{Class 1:} $20$ problems with  $20 < M < 100$, $50 < m < 500$, $5 < n_i < 100$    and $\gamma_A = 0.5$.
\item\textit{Class 2:} $20$ problems with $100 < M < 1000$, $100 < m < 600$, $10 < n_i < 50$ and $\gamma_A = 0.1$.
\item\textit{Class 3:} $10$ problems with $1000 < M < 2000$, $500 < m < 1000$, $100 < n_i < 200$ and $\gamma_A = 0.05$.
\end{itemize}

\vskip0.1cm
\noindent\textit{Scenarios. }
We considered two different scenarios:
\newline\textit{Scenario I:}
In this scenario, we aimed at comparing Algorithms \ref{alg:A1} and \ref{alg:A2}, \texttt{2pDecompAlg}, \texttt{pdDecompAlg}, \texttt{ADMM-v1} and
\texttt{EPBDM}, where we generated the values of $Q$ relatively small.  More precisely, we chose $[l_Q, u_Q] = [-0.1, 0.1]$, $[l_A, u_A] = [-1, 1]$ and $r_{x_0}
= 2$.
\newline
\textit{Scenario II:}
The second scenario aimed at testing the affect of the matrix $A$ and the update rule of the penalty parameter to the performance of \texttt{ADMM}.
We chose $[l_Q, u_Q] = [-1, 1]$, $[l_A, u_A] = [-5, 5]$ and $r_{x_0} = 5$.

\vskip0.1cm
\noindent\textit{Results. }
In the first scenario, the size of the problems satisfied $23 \leq M \leq 1992$, $95\leq m \leq 991$ and $1111 \leq n \leq 297818$.
The performance profiles of the six algorithms are plotted in Figure \ref{fig:perf_qp_exam02} with respect to the number of iterations and computational time.
\begin{figure}[!ht]
% \vskip-0.2cm
\centerline{\includegraphics[angle=0,height=3.8cm,width=12.3cm]{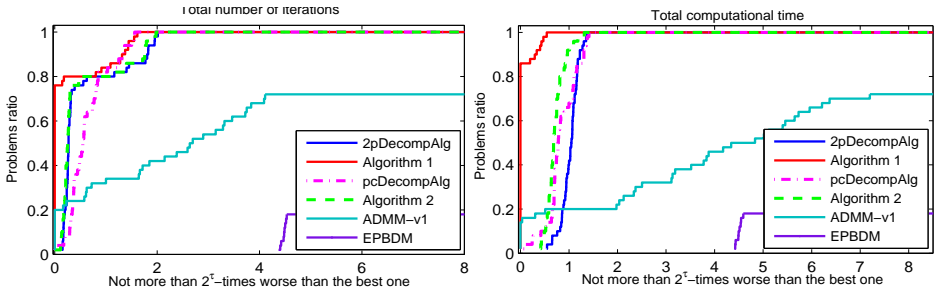}}
\caption{Performance profiles in $\log_2$ scale for Scenario I by using \texttt{IpOpt}: Left-Number of iterations, Right-Computational time.}
\label{fig:perf_qp_exam02}
\vskip-0.2cm
\end{figure}

From these performance profiles, we can observe that the Algorithm \ref{alg:A1}, Algorithm \ref{alg:A2}, \\  \texttt{2pDecompAlg} and \texttt{pdDecompAlg}
converged for all problems.  \texttt{ADMM-v1} was successful in solving $36/50$ ($72.00\%$) problems while \texttt{EPBDM} could only solve $9/50$ $(18.00\%)$
problems. It shows that Algorithm \ref{alg:A1} is the best one in terms of number of iterations. It could solve up to $38/50$ ($76.00\%$) problems
with the best performance. \texttt{ADMM-v1} solved $10/50$ ($20.00\%$) problems with the best performance, while this ratio was only $2/50$ ($4.00\%$) and
$1/50$ ($2.00\%$) in \texttt{pdDecompAlg} and Algorithm \ref{alg:A2}, respectively. If we compare the computational time then Algorithm \ref{alg:A1} is the best
one. It could solve up to $43/50$ ($86.00\%$) problems with the best performance. \texttt{ADMM-v1} solved $7/50$ ($14.00\%$) problems with the best performance.

Since the performance of Algorithms \ref{alg:A1} and \ref{alg:A2}, \texttt{2pDecompAlg}, \texttt{pdDecompAlg} and \texttt{ADMM} are relatively comparable, we
tested Algorithms \ref{alg:A1} and \ref{alg:A2}, \texttt{2pDecompAlg}, \texttt{pdDecompAlg}, \texttt{ADMM-v1}, \texttt{ADMM-v2} and
\texttt{ADMM-v3} on a collection of $n_p = 50$ problems in the second scenario.
The performance profiles of these algorithms are shown in Figure \ref{fig:perf_qp_exam01}.
\begin{figure}[!ht]
\vskip-0.2cm
\centerline{\includegraphics[angle=0,height=3.8cm,width=12.3cm]{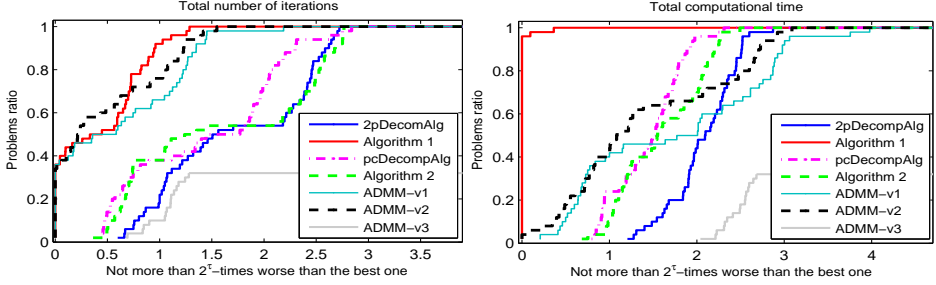}}
\caption{Performance profiles in $\log_2$ scale for Scenario II by using \texttt{Cplex} with Simplex method: Left-Number of iterations, Right-Computational
time.}\label{fig:perf_qp_exam01}
\vskip-0.5cm
\end{figure}
From these performance profiles we can observe the following:
\begin{itemize}
\item The six first algorithms were successful in solving all problems, while \texttt{ADMM-v3} could only solve $16/50$ ($32\%$) problems.
\item Algorithm \ref{alg:A1} and \texttt{ADMM-v1} is the best one in terms of number of iterations. It both solved $18/50$ ($36\%$) problems with the
best performance. This ratio is $17/50$ $(34\%)$ in \texttt{ADMM-v2}.
\item Algorithm \ref{alg:A1} is the best one in terms of computational time. It could solve $48/50$ $(96\%)$ the problems with the best performance, while this
quantity is $2/50~(4\%)$ in \texttt{ADMM-v2}.
\end{itemize}

%%%%%%%%%%%%%%%%%%%%%%%%%%%%%%%%%%%%%%%%%%%%%%%%%%%%%%%%%%%%%%%%%%%%%%
%+ 6.4. Nonlinear smooth separable convex programming
%%%%%%%%%%%%%%%%%%%%%%%%%%%%%%%%%%%%%%%%%%%%%%%%%%%%%%%%%%%%%%%%%%%%%%
\vskip0.1cm
\noindent\textbf{6.4. Nonlinear smooth separable convex programming.}
We consider the following nonlinear, smooth and separable convex programming problem:
\begin{equation}\label{eq:resource_allocation}
\left\{\begin{array}{cl}
\displaystyle\min_{x_i\in \mathbb{R}^{n_i}} &\Big\{\phi(x) := \displaystyle\sum_{i=1}^M\frac{1}{2}(x_i-x_i^0)Q_i(x_i - x_i^0) - w_i\ln(1 + b_i^Tx_i) \Big\},\\
\textrm{s.t.} &\displaystyle\sum_{i=1}^MA_ix_i = b,\\
&x_i\succeq 0, ~~i=1,\dots, M.
\end{array}\right.
\end{equation}
Here, $Q_i$ is a positive semidefinite and $x^i_0$ is given vector, $i=1,\dots, M$.

\vskip0.1cm
\noindent\textit{Problem generation. }
In this example, we generated a collection of $n_p = 50$ test problems as follows.
Matrix $Q_i$ is diagonal and was generated randomly in $[l_Q, u_Q]$.
Matrix $A_i$ was generated randomly in $[l_A, u_A]$ with the density $\gamma_A$.
Vectors $b_i$ and $w_i$ were generated randomly in $[l_b, u_b]$ and $[0, 1]$, respectively, such that $w_i\geq 0$ and $\sum_{i=1}^Mw_i = 1$.
Vector $b := \sum_{i=1}^MA_ix_i^0$ for a given $x_i^0$ in $[0, r_{x_0}]$.
The size of the problems was generated randomly based on the following rules:
\begin{itemize}
\item\textit{Class 1:} $10$ problems with $20 < M < 50$, $50 < m < 100$, $10 < n_i < 50$ and $\gamma_A = 1.0$.
\item\textit{Class 2:} $10$ problems with $50 < M < 250$, $100 < m < 200$, $20 < n_i < 50$ and $\gamma_A = 0.5$.
\item\textit{Class 3:} $10$ problems with $250 < M < 1000$, $100 < m < 500$, $50 < n_i < 100$ and $\gamma_A = 0.1$.
\item\textit{Class 4:} $10$ problems with $1000 < M < 5000$, $500 < m < 1000$, $50 < n_i < 100$ and $\gamma_A = 0.05$.
\item\textit{Class 5:} $10$ problems with $5000 < M < 10000$, $500 < m < 1000$, $50 < n_i < 100$ and $\gamma_A = 0.01$.
\end{itemize}

\vskip0.1cm
\noindent\textit{Scenarios. }
We also considered two different scenarios as in the previous example:
\newline
\textit{Scenario I: }
Similar to the previous example, with this scenario, we aimed at comparing Algorithms \ref{alg:A1} and \ref{alg:A2}, \texttt{2pDecompAlg}, \texttt{pdDecompAlg},
\texttt{ADMM-v1}, \texttt{PCBDM} and \texttt{EPBDM}. In this scenario, we chose: $[l_Q, u_Q] \equiv [-0.01, 0.01]$, $[l_b, u_b] \equiv [0, 100]$, $[l_A, u_A]
\equiv [-1, 1]$ and $r_{x_0} = 1$.
\newline
\textit{Scenario II:  }
In this scenario, we only tested two first variants of \texttt{ADMM} and compared them with the four first algorithms. Here, we chose $[l_Q, u_Q] \equiv
[0.0, 0.0]$ (i.e. without quadratic term), $[l_b, u_b] \equiv [0, 100]$, $[l_A, u_A] \equiv [-1, 1]$ and $r_{x_0} = 10$.

\vskip0.1cm
\noindent\textit{Results. }
For \textit{Scenario I}, we see that the size of the problems is in $20 \leq M \leq 9938$, $50 \leq m \leq 999$ and $695 \leq n \leq 741646$.
The performance profiles of the algorithms are plotted in Figure \ref{fig:perf_scp2}.
\begin{figure}[ht]
\vskip-0.2cm
\centerline{\includegraphics[angle=0,height=3.8cm,width=12.3cm]{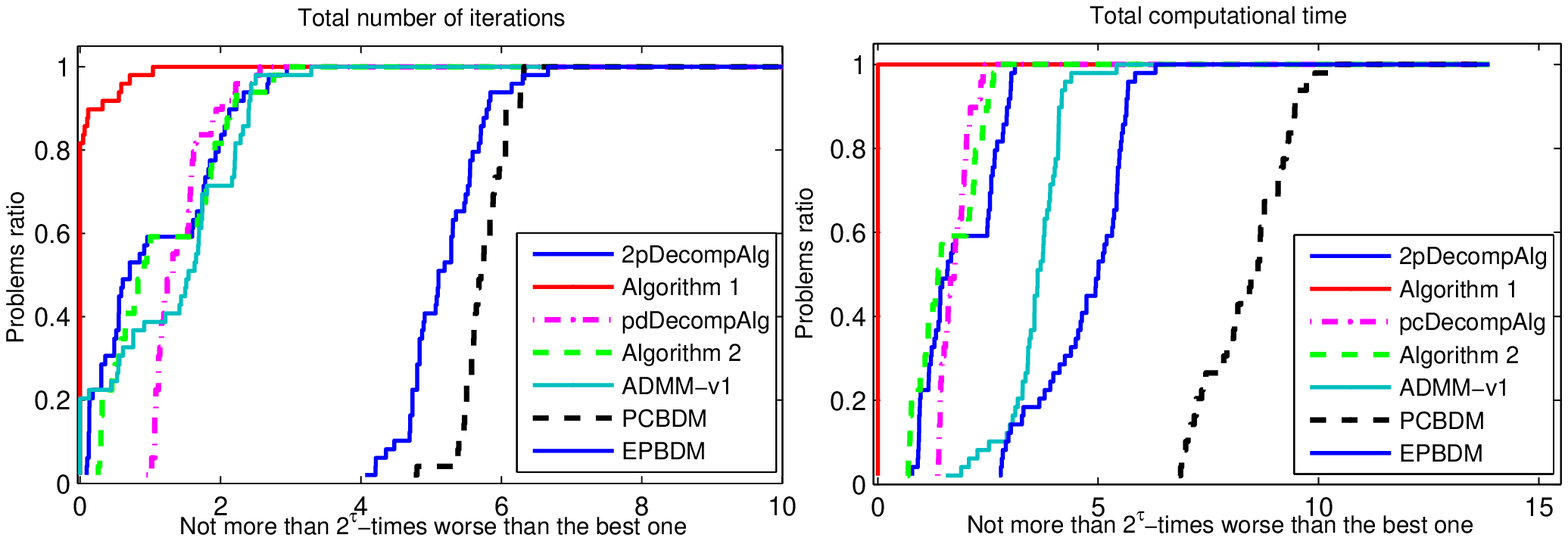}}
\caption{Performance profiles on \texttt{Scenario II} in $\log_2$ scale by using \texttt{IpOpt}: Left-Number of iterations, Right-Computational time.}
\label{fig:perf_scp2}
\vskip-0.3cm
\end{figure}
The results on this collection shows that Algorithm \ref{alg:A1} is the best one in terms of number of iterations.
It could solve up to $41/50$ ($82\%$) problems with the best performance, while \texttt{ADMM-v1} solved $10/50$ ($20\%$) problems with the best performance.
Algorithm \ref{alg:A1} is also the best one in terms of computational time. It could solve $50/50$ ($100\%$) problems with the best performance. \texttt{PCBDM}
was very slow compared to the rest in this scenario.

For \textit{Scenario II}, the size of the problems was varying in  $20 \leq M \leq 9200$, $50 \leq m \leq 946$ and $695 \leq n \leq 684468$.
The performance profiles of the tested algorithms are plotted in Figure \ref{fig:perf_scp1}.
\begin{figure}[ht]
\vskip-0.0cm
\centerline{\includegraphics[angle=0,height=3.8cm,width=12.3cm]{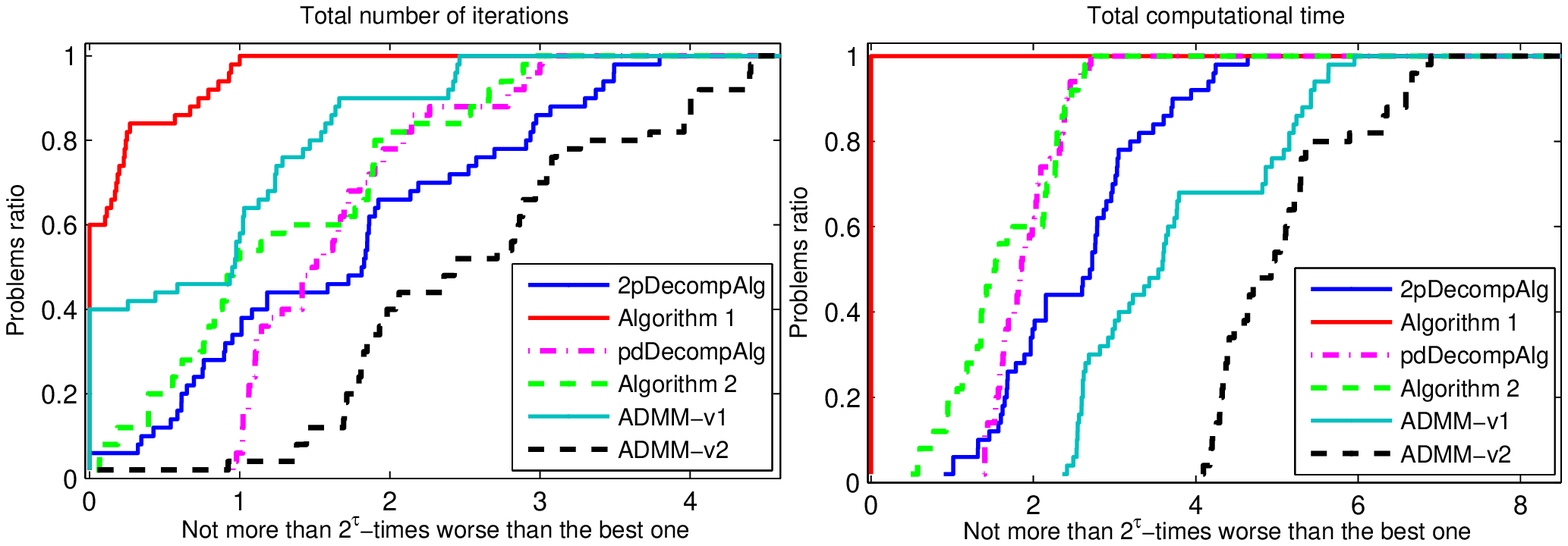}}
\caption{Performance profiles in $\log_2$ scale for \texttt{Scenario I} by using \texttt{IpOpt}: Left-Number of iterations, Right-Computational time.}
\label{fig:perf_scp1}
\vskip-0.3cm
\end{figure}
We can see from these performance profiles that Algorithm \ref{alg:A1} is the best one in terms of number of iterations.
It could solve up to $30/50$ ($60\%$) problems with the best performance, while this number were $3/50$ ($6\%$) and
$20/50$ ($40\%$) problems in \texttt{2pDecompAlg} and \texttt{ADMM-v1}, respectively.
Algorithm \ref{alg:A1} was also the best one in terms of computational time. It solved all problems with the best performance. \texttt{ADMM-v2} was slow
compared to the rest in this scenario.

From the above two numerical tests, we can observe that Algorithm \ref{alg:A1} performs well compared to the rest in terms of computational time due to
its low cost per iteration. ADMM encounters some difficulty regarding the choice of the penalty parameter as well as the effect of matrix $A$. Theoretically,
PCBDM has the same worst-case complexity bound as Algorithms \ref{alg:A1} and \eqref{alg:A2}. However, its performance is quite poor. This happens due to the
choice of the Lipschitz constant $L_{\mathrm{A}}$ of the gradient of the dual function and the evaluation of the quantity $D_X$.

%%%%%%%%%%%%%%%%%%%%%%%%%%%%%%%%%%%%%%%%%%%%%%%%%%%%%%%%%%%%%%%%%%%%%%%%%%%%%%%%%%%%%%%%%%%%%%%%%
%+ 8. Conclusion.
%%%%%%%%%%%%%%%%%%%%%%%%%%%%%%%%%%%%%%%%%%%%%%%%%%%%%%%%%%%%%%%%%%%%%%%%%%%%%%%%%%%%%%%%%%%%%%%%%
\section{Concluding remarks}\label{sec:conclusion}
We have proposed a  new decomposition algorithm based on the dual decomposition and excessive gap techniques.
The new algorithm requires to perform only one primal step which can be parallelized efficiently, and two dual steps.
Consequently, the computational complexity of this algorithm is very similar to other dual based decomposition algorithms from the literature,  but with a
better theoretical rate of convergence. Moreover, the algorithm automatically updates both smoothness parameters at each iteration.
We notice that the dual steps are only matrix-vector multiplications, which can be done efficiently with a low computational cost in practice.
Furthermore, we allow one to solve the primal convex subproblem of each component up to a given accuracy, which is always the case in any
practical implementation.
An inexact switching variant of Algorithm  \ref{alg:A1} has also been presented.
Apart from the inexactness, this variant allows one to simultaneously update both smoothness parameters instead of switching them.
Moreover, it improves the disadvantage of Algorithm \ref{alg:A1} when the constant $\alpha^{*}$ in Theorem \ref{th:alg_scheme} is relatively
small, though it did not outperform Algorithm \ref{alg:A1} in the numerical tests.
The worst-case complexity of both new algorithms is at most $O(1/\varepsilon)$ for a given tolerance $\varepsilon> 0$.
Preliminary numerical tests show that both algorithms outperforms other related existing algorithms from the literature.

%////////////////////////////////////////////////////////////////////////////////////////////////////////////////////////////////////
%+ Acknowledgments.
%////////////////////////////////////////////////////////////////////////////////////////////////////////////////////////////////////
\vskip 0.2cm
\begin{footnotesize}
\noindent{\textbf{Acknowledgments.}}
This research was supported by Research Council KUL: CoE EF/05/006 Optimization in Engineering(OPTEC), GOA AMBioRICS,
IOF-SCORES4CHEM, several PhD/postdoc \& fellow grants; the Flemish Government via FWO: PhD/postdoc grants, projects G.0452.04, G.0499.04, G.0211.05,
G.0226.06, G.0321.06, G.0302.07, G.0320.08 (convex MPC), G.0558.08 (Robust MHE), G.0557.08, G.0588.09, research communities (ICCoS, ANMMM, MLDM) and via IWT:
PhD Grants, McKnow-E, Eureka-Flite+EU: ERNSI; FP7-HD-MPC (Collaborative Project STREP-grantnr. 223854), Contract Research: AMINAL, HIGHWIND, and Helmholtz
Gemeinschaft: viCERP; Austria: ACCM, and the Belgian Federal Science Policy Office: IUAP P6/04 (DYSCO, Dynamical systems, control and optimization, 2007-2011),
European Union FP7-EMBOCON under grant agreement no 248940; CNCS-UEFISCDI (project TE231, no. 19/11.08.2010); ANCS (project PN II, no. 80EU/2010); Sectoral
Operational Programme Human Resources Development 2007-2013 of the Romanian Ministry of Labor, Family and Social Protection through the Financial Agreement
POSDRU/89/1.5/S/62557.
\end{footnotesize}

%%%%%%%%%%%%%%%%%%%%%%%%%%%%%%%%%%%%%%%%%%%%%%%%%%%%%%%%%%%%%%%%%%%%%%%%%%%%%%%%%%%%%%%%%%%%%%%%%%%%%%%%%%
%+ 6. The proofs of technical lemmas
%%%%%%%%%%%%%%%%%%%%%%%%%%%%%%%%%%%%%%%%%%%%%%%%%%%%%%%%%%%%%%%%%%%%%%%%%%%%%%%%%%%%%%%%%%%%%%%%%%%%%%%%%%
\appendix
\section{The details of the proofs}\label{sec:proofs}
In this appendix we provide the full proof of Theorem \ref{th:alg_scheme}, Theorem \ref{th:alg_scheme2}, Lemma \ref{le:init_point} and Lemma
\ref{le:choice_of_tau}.

%+ ******************************************************************
%+ 6.1. The proof of Theorem 3.1.
%+ ******************************************************************
\vskip0.2cm
\noindent\textit{A.1. The proof of Theorem \ref{th:alg_scheme}. }
%+ Proof of Theorem 3.1.
Let us denote by $\bar{y}^2 := y^{*}(\bar{x}; \beta_2)$, $x^1 := x^{*}(\hat{y}; \beta_1)$ and $\tilde{x}^1 = \tilde{x}^{*}(\hat{y}; \beta_1)$.
From the definition of $f$, the second line of \eqref{eq:alg_scheme} and \eqref{eq:beta_update}, we have
\begin{align}\label{eq:th31_proof1}
f(\bar{x}^{+}; \beta_2^{+}) &:= \phi(\bar{x}^{+}) + \psi(\bar{x}^{+};\beta_2^{+})
\overset{\tiny{\mathrm{line}~2\eqref{eq:alg_scheme}}}{{\!\!\!\!\!\!\!\!\!\!}={\!\!\!\!\!\!}}
\phi((1 \!-\! \tau)\bar{x} \!+\! \tau \tilde{x}^{1})  +  \max_{y\in\mathbb{R}^m}\left\{\!\left[A((1 \!-\!\tau)\bar{x} \!+\! \tau\tilde{x}^{1}) \!-\! b\right]^Ty
\!-\! \frac{\beta_2^{+}}{2}\norm{y}^2\right\} \nonumber\\
& \overset{\tiny{\phi-\mathrm{convex}+\eqref{eq:beta_update}}}{\leq} \max_{y\in\mathbb{R}^m}\Big\{ \! (1 \!-\! \tau)\big[\phi(\bar{x}) \!+\! (A\bar{x}
\!-\! b)^Ty \!-\! \frac{\beta_2}{2}\norm{y}^2\big]_{[1]}  +   \tau\left[\phi( \tilde{x}^1) \!+\! (A\tilde{x}^1 \!-\! b)^Ty \right]_{[2]} \Big\}.
\end{align}
Now, we estimate two terms in the last line of \eqref{eq:th31_proof1}.
First we note that $a^Ty - \frac{\beta}{2}\norm{y}^2 = \frac{1}{2\beta}\norm{a}^2 - \frac{\beta}{2}\norm{y-\frac{1}{\beta}a}^2$ for any vectors $a$ and $y$ and
$\beta > 0$. Moreover, since $\bar{y}^2$ is the solution of the strongly concave maximization \eqref{eq:psi_fun} with a concavity parameter $\beta_2$, we can
estimate
\begin{eqnarray}\label{eq:th31_proof2}
[\cdot]_{[1]} &&:= \phi(\bar{x}) + (A\bar{x} - b)^Ty - \frac{\beta_2}{2}\norm{y}^2 = \phi(\bar{x}) + \frac{1}{2\beta_2}\norm{A\bar{x} - b}^2  -
\frac{\beta_2}{2}\norm{y -
\bar{y}^2}^2 \nonumber\\
&& = \phi(\bar{x}) + \psi(\bar{x};\beta_2) - \frac{\beta_2}{2}\norm{y - \bar{y}^2}^2 \overset{\tiny{\eqref{eq:smoothed_phi}}}{=}
f(\bar{x}; \beta_2) - \frac{\beta_2}{2}\norm{y - \bar{y}^2}^2 \\
&& \overset{\tiny{\eqref{eq:excessive_gap_inexact}}}{\leq} g(\bar{y}; \beta_1) - \frac{\beta_2}{2}\norm{y - \bar{y}^2}^2 + \delta
\overset{\tiny{g(\cdot;\beta_1)-\mathrm{concave}}}{\leq}
g(\hat{y}; \beta_1) + \nabla_yg(\hat{y}; \beta_1)^T(\bar{y} - \hat{y}) - \frac{\beta_2}{2}\norm{y - \bar{y}^2}^2 + \delta \nonumber\\
&& \overset{\eqref{eq:app_deriv_g}}{\leq} {\!\!} g(\hat{y}; \beta_1) \!+\! \widetilde{\nabla}_yg(\hat{y}; \beta_1)^T(\bar{y} \!-\! \hat{y})\!-\!
\frac{\beta_2}{2}\norm{y
\!-\! \bar{y}^2}^2 \!+\! (\bar{y} \!-\! \hat{y})^TA(x^1 \!-\! \tilde{x}^1) \!+\! \delta. \nonumber
\end{eqnarray}
Alternatively, by using \eqref{eq:accuracy_constants}, the second term $[\cdot]_{[2]}$ can be estimated as
\begin{eqnarray}\label{eq:th31_proof3}
[\cdot]_{[2]} && := \phi(\tilde{x}^1) + (A\tilde{x}^1-b)^Ty \nonumber\\
&& = \phi(\tilde{x}^1) + (A\tilde{x}^1 - b)^T\hat{y} + \beta_1p_X(\tilde{x}^1) + (A\tilde{x}^1 - b)^T(y - \hat{y}) - \beta_1p_X(\tilde{x}^1) \nonumber\\
[-1.6ex]\\[-1.6ex]
&&\overset{\tiny\eqref{eq:inexactness_xi}}{\leq}{\!\!\!}
\phi(x^1) \!+\! (Ax^1 \!-\! b)^T\hat{y} \!+\! \beta_1p_X(x^1) \!+\! (A\tilde{x}^1 \!-\! b)^T\!(y \!-\! \hat{y}) \!-\!  \beta_1p_X(\tilde{x}^1) \!+\!
\frac{\beta_1}{2}\varepsilon_{[\sigma]}^2 \nonumber\\
&&\overset{\tiny{\eqref{eq:smoothed_gi}+\mathrm{Lemma}~\ref{le:dual_smoothed_estimates}}}{=}
 g(\hat{y}; \beta_1) \!+\!
\widetilde{\nabla}_yg(\hat{y}; \beta_1)^T(y \!-\! \hat{y})  \!-\! \beta_1p_X(\tilde{x}^1) \!+\! \frac{\beta_1}{2}\varepsilon_{[\sigma]}^2.
\nonumber
\end{eqnarray}
Next,  we consider the point $u := \bar{y} + \tau(y - \bar{y})$ with $\tau\in (0, 1)$. On the one hand, it is easy to see that if $y\in\mathbb{R}^m$ then $u \in
\mathbb{R}^m$. Moreover, we have
\begin{equation}\label{eq:th31_proof4}
\begin{cases}
(1-\tau)(\bar{y} - \hat{y}) + \tau(y - \hat{y}) = \bar{y} + \tau(y - \bar{y}) - \hat{y} = u - \hat{y},\\
u - \hat{y} = u - (1-\tau)\bar{y} - \tau\bar{y}^2 = \tau(y - \bar{y}^2). \\
\end{cases}
\end{equation}
On the other hand, it follows from \eqref{eq:alg_scheme_cond} that
\begin{equation}\label{eq:th31_proof5}
\frac{(1-\tau)}{\tau^2}\beta_2 \geq \frac{L_{\mathrm{A}}}{\beta_1}
\overset{\tiny\mathrm{Lemma}~\ref{le:dual_smoothed_estimates}}{\geq}
L^g(\beta_1), ~i=1,\dots, M.
\end{equation}
By substituting \eqref{eq:th31_proof2} and \eqref{eq:th31_proof3} into \eqref{eq:th31_proof1} and then using \eqref{eq:th31_proof4} and \eqref{eq:th31_proof5},
we conclude that
\begin{eqnarray}\label{eq:th31_proof6}
f(\bar{x}^{+}; \beta_2^{+}) && \leq \max_{y\in\mathrm{R}^m}\left\{ (1-\tau)[\cdot]_{[1]} + \tau [\cdot]_{[2]} \right\}  \nonumber\\
&& \overset{\tiny{\eqref{eq:th31_proof2}+\eqref{eq:th31_proof3}}}{\leq} \max_{y\in\mathbb{R}^m}\Big\{(1-\tau)g(\hat{y};\beta_1) + \tau g(\hat{y};\beta_1)  +
\widetilde{\nabla}_yg(\hat{y};\beta_1)^T\left[(1-\tau)(\bar{y}-\hat{y}) + \tau(y-\hat{y})\right] \nonumber\\
&& - \frac{(1-\tau)\beta_2}{2}\norm{y-\bar{y}^2}^2\Big\} - \tau\beta_1p_X(\tilde{x}^1) + 0.5\tau\beta_1\varepsilon_{[\sigma]}^2 + (1-\tau)\delta +
(1-\tau)(\bar{y}-\hat{y})^TA(x^1-\tilde{x}^1) \nonumber\\
&&{\!\!\!\!\!\!\!\!\!}\overset{\tiny{\eqref{eq:th31_proof4}}}{=}{\!\!\!\!} \left[\max_{u\in\mathbb{R}^m}\left\{ g(\hat{y};\beta_1) +
\widetilde{\nabla}_yg(\hat{y};\beta_1)^T(u - \hat{y}) - \frac{(1-\tau)\beta_2}{2\tau^2}\norm{u-\hat{y}}^2\right\} \right]_{[3]}\nonumber\\
&& + \left[0.5\tau\beta_1\varepsilon_{[\sigma]}^2 + (1-\tau)\delta + (1-\tau)(\bar{y}-\hat{y})^TA(x^1-\tilde{x}^1) -
\tau\beta_1p_X(\tilde{x}^1)\right]_{[4]}.
\end{eqnarray}
Let us consider the first term $[\cdot]_{[3]}$ of \eqref{eq:th31_proof6}. We see that
\begin{eqnarray}\label{eq:th31_proof6a}
[\cdot]_{[3]} &&= \max_{u\in\mathbb{R}^m}\left\{ g(\hat{y};\beta_1) + \widetilde{\nabla}_yg(\hat{y};\beta_1)^T(u - \hat{y}) -
\frac{(1-\tau)\beta_2}{2\tau^2}\norm{u-\hat{y}}^2\right\} \nonumber\\
&&\overset{\tiny{\eqref{eq:th31_proof5}}}{\leq}{\!\!\!\!\!}
 \max_{u\in\mathbb{R}^m}\left\{ g(\hat{y};\beta_1) + \widetilde{\nabla}_yg(\hat{y};\beta_1)^T(u - \hat{y}) -
\frac{L^g(\beta_1)}{2}\norm{u-\hat{y}}^2\right\}\nonumber\\
&&{\!\!\!\!\!\!\!\!\!\!\!\!}\overset{\tiny{\eqref{eq:gradient_mapping_ex}+\eqref{eq:alg_scheme}(\mathrm{line}~3)}}{=}{\!\!\!\!\!\!\!\!\!\!}
g(\hat{y}; \beta_1) + \widetilde{\nabla}_yg(\hat{y};\beta_1)^T(\bar{y}^{+} - \hat{y}) - \frac{L^g(\beta_1)}{2}\norm{\bar{y}^{+} - \hat{y}}^2\nonumber\\
&&\overset{\tiny{\eqref{eq:app_deriv_g_est}}}{=} {\!\!\!\!}
g(\hat{y};\beta_1) \!+\! \nabla_yg(\hat{y};\beta_1)^T\!\! (\bar{y}^{+} \!\!-\! \hat{y}) \!-\! \frac{L^g(\beta_1)}{2}\norm{\bar{y}^{+} \!\!-\! \hat{y}}^2 \!+\!
(\bar{y}^{+}
\!-\! \hat{y})^T\!\!A(\tilde{x}^1 \!-\! x^1) \\
&&\overset{\tiny\eqref{eq:g_Lipschitz_bound}}{\leq}
g(\bar{y}^{+}; \beta_1) + (\bar{y}^{+} - \hat{y})^TA(x^1-\tilde{x}^1)\nonumber\\
&&\overset{\tiny\eqref{eq:g_beta1_Lipschitz_bound}}{\leq}
g(\bar{y}^{+}; \beta_1^{+}) + (\beta_1-\beta_1^{+})p_X(x^{*}(\bar{y}^{+};\beta_1^{+})) + (\bar{y}^{+} - \hat{y})^TA(x^1-\tilde{x}^1)\nonumber\\
&&\overset{\tiny\eqref{eq:beta_update}+\eqref{eq:D_X_p_X}}{\leq}
g(\bar{y}^{+}; \beta_1^{+}) + \left[\tilde{\alpha}\tau\beta_1D_X + (\bar{y}^{+} - \hat{y})^TA(x^1-\tilde{x}^1)\right]_{[5]}\nonumber.
\end{eqnarray}
In order to estimate the term $[\cdot]_{[4]} + [\cdot]_{[5]}$, we can observe that
\begin{align*}
(\bar{y}^{+} - \hat{y}) - (1-\tau)(\hat{y}-\bar{y}) &\overset{\tiny\eqref{eq:alg_scheme}\mathrm{line}~1}{=} L^g(\beta_1)^{-1}(A\tilde{x}^1-b) +
(1-\tau)\tau(\bar{y}^2 - \bar{y})\nonumber\\
& = L^g(\beta_1)^{-1}A(\tilde{x}^1-x^c) + L^g(\beta_1)^{-1}(Ax^c-b) - (1-\tau)\tau\bar{y}\nonumber\\
& + \beta_2^{-1}(1-\tau)\tau A(\bar{x}-x^c) + \beta_2^{-1}(1-\tau)\tau(Ax^c-b)\nonumber,
\end{align*}
which leads to
\begin{align}\label{eq:th31_proof6c}
A^T\left[(\bar{y}^{+} \!-\! \hat{y}) \!-\! (1 \!-\! \tau)(\hat{y} \!-\! \bar{y})\right] & \leq
L_{\mathrm{A}}^{-1}\beta_1\norm{A}^2\norm{\tilde{x}^1-x^c} + L_{\mathrm{A}}^{-1}\beta_1\norm{A^T(Ax^c-b)} +
\frac{(1-\tau)\tau}{\beta_2}\norm{A}^2\norm{\bar{x}-x^c} \nonumber\\
& + \beta_2^{-1}(1-\tau)\tau\norm{A^T(Ax^c-b)} + (1-\tau)\tau\norm{A}\norm{\bar{y}} .
\end{align}
Note that similar to \eqref{eq:lm31_proof4b}, we have $\norm{\tilde{x}^1 - x^c} \leq D_{\sigma}$ and $\norm{\bar{x} - x^c} \leq D_{\sigma}$. By
substituting these estimates into \eqref{eq:th31_proof6c} and using the definitions of $[\cdot]_{[4]}$ and $[\cdot]_{[5]}$ we have
\begin{equation}\label{eq:th31_proof6d}
[\cdot]_{[4]} + [\cdot]_{[5]} \leq (1-\tau)\delta + \frac{\tau\beta_1}{2}\varepsilon_{[\sigma]}^2 + \tau\beta_1(\tilde{\alpha} D_X -
p_X(\tilde{x}^1)) + \left[\frac{\beta_1}{L_{\mathrm{A}}}C_d + (1-\tau)\tau\left(\frac{C_d}{\beta_2} + \norm{A}\norm{\bar{y}}\right)\right]\varepsilon_{[1]}.
\end{equation}
By combining \eqref{eq:th31_proof6}, \eqref{eq:th31_proof6a} and \eqref{eq:th31_proof6d} and noting that $\tilde{\alpha} D_X - p_X(\tilde{x}^1)\leq 0$, we
obtain
\begin{eqnarray}\label{eq:th31_proof7}
f(\bar{x}^{+};\beta_2^{+}) && \leq g(\bar{y}^{+};\beta_1^{+}) + \tau\beta_1(\tilde{\alpha} D_X - p_X(\tilde{x}^1)) + (1-\tau)\delta + \eta(\tau,
\beta_1,\beta_2,
\bar{y}, \varepsilon)\nonumber\\
&&\leq g(\bar{y}^{+};\beta_1^{+}) + (1-\tau)\delta + \eta(\tau, \beta_1,\beta_2, \bar{y},
\varepsilon)\nonumber\\
&&= g(\bar{y}^{+};\beta_1^{+}) + \delta_{+}\nonumber,
\end{eqnarray}
which is indeed  inequality \eqref{eq:excessive_gap_inexact} w.r.t. $\beta_1^{+}$, $\beta_2^{+}$ and $\delta_{+}$.
\Eproof
%+ End of the proof.

%+ ******************************************************************
%+ 6.2. The proof of Theorem 3.1.
%+ ******************************************************************
\vskip0.2cm
\noindent\textit{A.2. The proof of Theorem \ref{th:alg_scheme2}.}
Let us denote by $y^2_{+} = y^{*}(\hat{x};\beta_2^{+})$, $x^1 := x^{*}(\bar{y}; \beta_1)$, $\tilde{x}^1 := \tilde{x}^{*}(\bar{y};\beta_1)$,  $\bar{x}^{*+} :=
P(\hat{x}; \beta_2^{+})$  and $\norm{x-x^1}_{\sigma}^2 := \sum_{i=1}^M\sigma_i\norm{x_i-x_i^1}^2$.
From the definition of $g(\cdot; \beta_1)$, the second line of \eqref{eq:alg_scheme_cond_p} and $\beta_1^{+} = (1-\tau)\beta_1$ we have
\begin{align}\label{eq:thm41_proof1}
g(\bar{y}^{+}; \beta_1^{+}) &= \min_{x\in X}\left\{\phi(x) + (\bar{y}^{+})^T(Ax-b) + \beta_1^{+}p_X(x) \right\}
\nonumber\\
&\overset{\tiny{\mathrm{line~2}\eqref{eq:alg_scheme_cond_p}}}{=}\min_{x\in X}\Big\{(1-\tau)\left[\phi(x) + \bar{y}^T(Ax-b) +
\beta_1p_X(x) \right]_{[1]}  + \tau \left[ \phi(x) + (y^2_{+})^T(Ax-b)\right]_{[2]}\Big\}.
\end{align}
First, we estimate the term $[\cdot]_{[1]}$ in \eqref{eq:thm41_proof1}. Since each component of the function in $[\cdot]_{[1]}$ is strongly convex w.r.t. $x_i$
with a convexity parameter $\beta_1\sigma_i > 0$ for $i=1,\dots, M$, by using the optimality condition, one can show that
\begin{eqnarray}\label{eq:thm41_proof2}
[\cdot]_{[1]}  &&\overset{\tiny{\eqref{eq:smooth_dual_sol}}}{\geq}  \min_{x\in X}\left\{\phi(x) + \bar{y}^T(Ax-b) + \beta_1p_X(x)\right\} +
\frac{\beta_1}{2}\norm{x-x^1}_{\sigma}^2 \overset{\tiny\eqref{eq:smoothed_gi}}{=} g(\bar{y}; \beta_1) + \frac{\beta_1}{2}\norm{x - x^1}_{\sigma}^2 \nonumber\\
[-1.8ex]\\[-1.8ex]
&&\overset{\tiny{\eqref{eq:excessive_gap_inexact}}}{\geq} f(\bar{x}; \beta_2) + \frac{\beta_1}{2}\norm{x-x^1}_{\sigma}^2 - \delta
\nonumber.
\end{eqnarray}
Moreover, since $\psi(\bar{x};\beta_2) = \frac{1}{2\beta_2}\norm{A\bar{x}-b}^2 = \frac{(1-\tau)}{2\beta_2^{+}}\norm{A\bar{x}-b}^2 =
(1-\tau)\psi(\bar{x};\beta_2^{+})$, by substituting this relation into \eqref{eq:thm41_proof2b} we obtain
\begin{eqnarray}\label{eq:thm41_proof2b}
[\cdot]_{[1]} &&\geq \phi(\bar{x}) + \psi(\bar{x}; \beta_2) + \frac{\beta_1}{2}\norm{x-x^1}_{\sigma}^2  - \delta \nonumber\\
&& =  \phi(\bar{x}) + \psi(\bar{x}; \beta_2^{+}) - \tau\psi(\bar{x};\beta_2^{+}) + \frac{\beta_1}{2}\norm{x-x^1}_{\sigma}^2 - \delta \\
&&{\!\!\!\!\!}\overset{\tiny\mathrm{def.}~\psi}{\geq}
\phi(\bar{x}) \!+\! \psi(x^2; \beta_2^{+}) \!+\! \nabla_x{\psi}(x^2; \beta_2^{+})^T(\bar{x} \!-\! x^2) \!+\!
\frac{\beta_1}{2}\norm{x \!-\! x^1}_{\sigma}^2 \!-\! \delta + \frac{1}{2\beta_2^{+}}\norm{A(\bar{x}-x^2)}^2 - \tau\psi(\bar{x};\beta_2^{+}). \nonumber
\end{eqnarray}
Here, the last inequality follows from the fact that $\psi(\bar{x};\beta_2^{+}) = \frac{1}{2\beta_2^{+}}\norm{A\bar{x} - b}^2$.

Next, we consider the term $[\cdot]_{[2]}$ of \eqref{eq:thm41_proof1}. We note that $y^2_{+} = \frac{1}{\beta_2^{+}}(Ax^2 - b)$. Hence,
\begin{eqnarray}\label{eq:thm41_proof3}
[\cdot]_{[2]} &&= \phi(x) + (y^2_{+})^TA(x - x^2) + (Ax^2 - b)^Ty^2_{+} \nonumber\\
&&\overset{\tiny\mathrm{Lemma}~\ref{le:primal_smoothed_estimates}}{=}
\phi(x) + \nabla_x{\psi}(x^2; \beta_2^{+})^T(x - x^2) + \frac{1}{\beta_2^{+}}\norm{Ax^2-b}^2\\
&& = \phi(x) + \psi(x^2;\beta_2^{+}) + \nabla_x{\psi}(x^2;\beta_2^{+})^T(x - x^2) + \psi(x^2; \beta_2^{+}).\nonumber
\end{eqnarray}
From the definitions of  $\norm{\cdot}_{\sigma}$, $D_{\sigma}$ and $\varepsilon_{[\sigma]}$ we have  $\norm{x-x^c}_{\sigma} \leq D_{\sigma}$,
$\norm{\tilde{x}^1-x^c}_{\sigma} \leq D_{\sigma}$ and $\norm{x^1-\tilde{x}^1}_{\sigma}\leq \varepsilon_{[\sigma]}$.
Moreover, $\norm{x-x^1}_{\sigma} \geq \abs{\norm{x-\tilde{x}^1}_{\sigma} - \norm{x^1 - \tilde{x}^1}_{\sigma}}$. By using these estimates, we can derive
\begin{eqnarray}\label{eq:thm41_proof3b}
\norm{x-x^1}_{\sigma}^2 &&\geq \left[ \norm{x-\tilde{x}^1}_{\sigma} - \norm{x^1 - \tilde{x}^1}_{\sigma}\right]^2 \nonumber\\
&&=\norm{x-\tilde{x}^1}_{\sigma}^2 - 2\norm{x-\tilde{x}^1}_{\sigma}\norm{x^1-\tilde{x}^1}_{\sigma} + \norm{x^1-\tilde{x}^1}_{\sigma}^2 \nonumber\\
[-1.5ex]\\[-1.5ex]
&&\geq \norm{x-\tilde{x}^1}_{\sigma}^2 - 2\norm{x^1-\tilde{x}^1}_{\sigma}\left[\norm{x - x^c}_{\sigma} + \norm{\tilde{x}^1 - x^c}_{\sigma}\right]\nonumber\\
&&\geq \norm{x-\tilde{x}^1}_{\sigma}^2 - 4D_{\sigma}\varepsilon_{[\sigma]}.\nonumber
\end{eqnarray}
Furthermore, the condition \eqref{eq:alg_scheme_cond_p} can be expressed as
\begin{equation}\label{eq:alg_scheme_cond_p_tmp}
\frac{(1-\tau)}{\tau^2}\beta_1\sigma_i \geq \frac{M\norm{A_i}^2}{(1-\tau)\beta_2} = L_i^{\psi}(\beta_2^{+}), ~i=1,\dots, M.
\end{equation}
By substituting \eqref{eq:thm41_proof2b}, \eqref{eq:thm41_proof3} and \eqref{eq:thm41_proof3b} into \eqref{eq:thm41_proof1} and then using
\eqref{eq:alg_scheme_cond_p_tmp} and note that  $\tau(x - \tilde{x}^1) = (1-\tau)\bar{x} + \tau x - x^2$, we obtain
\begin{eqnarray}\label{eq:thm41_proof4}
g(\bar{y}^{+};\beta_1^{+}) &&= \min_{x\in X}\left\{(1-\tau)[\cdot]_{[1]} + \tau [\cdot]_{[2]}\right\}\nonumber\\
&&{\!\!\!\!}\overset{\tiny{\eqref{eq:thm41_proof2}+\eqref{eq:thm41_proof3}}}{\geq}{\!\!\!\!} \min_{x\in X}\Big\{
(1 \!-\! \tau)\phi(\bar{x}) \!+\! \tau\phi(x) \!+\! \nabla{\psi}(x^2;\beta_2^{+})^T\left[(1 \!-\! \tau)(\bar{x} \!-\! x^2) \!+\! \tau(x \!-\!
x^2)\right] + \frac{(1-\tau)\beta_1}{2}\norm{x-x^1}^2_{\sigma}\Big\} \nonumber\\
&&  - (1-\tau)\delta  + \left[\tau\psi(x^2;\beta_2^{+}) - (1-\tau)\tau\psi(\bar{x};\beta_2^{+}) +
\frac{(1-\tau)}{2\beta_2^{+}}\norm{A(\bar{x}-x^2)}^2\right]_{[3]}.
\end{eqnarray}
We further estimate \eqref{eq:thm41_proof4} as follows
\begin{align}\label{eq:thm41_proof4b}
g(\bar{y}^{+};\beta_1^{+})& \overset{\tiny{\phi-\mathrm{convex}}}{\geq} \min_{x\in X}\Big\{\phi((1 \!-\! \tau)\bar{x} \!+\! \tau x) \!+\!
\nabla{\psi}(x^2;\beta_2^{+})((1 \!-\! \tau)\bar{x} \!+\! \tau x \!-\! x^2)\nonumber\\
& + \frac{(1-\tau)\beta_1}{2}\norm{x-\tilde{x}^1}^2_{\sigma}\Big\} + [\cdot]_{[3]} - (1-\tau)\delta -  2(1-\tau)\beta_1D_{\sigma}\varepsilon_{[\sigma]}
\nonumber\\
& \overset{\tiny\eqref{eq:thm41_proof3b}}{\geq} {\!\!\!\!\!\!\!} \min_{u := (1 \!-\! \tau)\bar{x} \!+\! \tau x \in X}{\!\!\!}\Big\{\phi(u) \!+\!
\psi(x^2;\beta_2^{+}) \!+\! \nabla{\psi}(x^2;\beta_2^{+})(u \!-\! x^2) \!+\! \frac{(1 \!-\! \tau)\beta_1}{2\tau^2}\norm{u \!-\! x^2}^2_{\sigma}\Big\}
\nonumber\\
& - 2(1-\tau)\beta_1D_{\sigma}\varepsilon_{[\sigma]}  - (1-\tau)\delta + [\cdot]_{[3]}\nonumber\\
&{\!\!\!\!\!\!\!\!\!}\overset{\tiny\eqref{eq:alg_scheme_cond_p_tmp}}{\geq}\!\!\! \min_{u \in X}\!\Big\{\! \phi(u) \!+\! \psi(\hat{x};\beta_2^{+}) \!+\!
\nabla{\psi}(x^2;\beta_2^{+})(u \!-\! x^2) \!+\! \frac{L^{\psi}(\beta_2^{+})}{2}\norm{u \!-\! x^2}^2_{\sigma}\Big\} \!-\! 2(1 \!-\!
\tau)\beta_1D_{\sigma}\varepsilon_{[\sigma]}  \!-\! (1 \!-\! \tau)\delta \!+\! [\cdot]_{[3]}\nonumber\\
& =q(\bar{x}^{*+},y^2_{+};\beta_2^{+}) - 2(1-\tau)\beta_1D_{\sigma}\varepsilon_{[\sigma]}  - (1-\tau)\delta + [\cdot]_{[3]}\nonumber\\
& \overset{\tiny\eqref{eq:Pi_approx}}{\geq} q(\bar{x}^{+}, y^2_{+};\beta_2^{+}) - 2(1-\tau)\beta_1D_{\sigma}\varepsilon_{[\sigma]}  - (1-\tau)\delta +
[\cdot]_{[3]} - 0.5\varepsilon_{A}^2\nonumber\\
&\overset{\tiny{\eqref{eq:psi_Lipschitz_bound}}}{\geq} f(\bar{x}^{+}; \beta_2^{+}) - 2(1-\tau)\beta_1D_{\sigma}\varepsilon_{[\sigma]}  -
(1-\tau)\delta + [\cdot]_{[3]} - 0.5\varepsilon_{A}^2,
\end{align}
where $\varepsilon_A := [\sum_{i=1}^ML_i^{\psi}(\beta_2^{+})\varepsilon_i^2]^{1/2}$.

To complete the proof, we estimate $[\cdot]_{[3]}$ as follows
\begin{eqnarray}\label{eq:thm41_proof5}
[\cdot]_{[3]} && = \tau\psi(x^2;\beta_2^{+}) - \tau(1-\tau)\psi(\bar{x};\beta_2^{+}) + \frac{(1-\tau)}{2\beta_2^{+}}\norm{A(\bar{x}-x^2}^2 \nonumber\\
&& = \frac{1}{2\beta_2^{+}}\left[\tau\norm{Ax^2 - b}^2 - \tau(1-\tau)\norm{A\bar{x}-b}^2 + (1-\tau)\norm{A(\bar{x}-x^2)}^2\right]\\
&& = \frac{1}{2\beta_2^{+}}\norm{(Ax^2-b) + (1-\tau)(A\bar{x}-b)}^2  \geq 0.\nonumber
\end{eqnarray}
By substituting \eqref{eq:thm41_proof5} into \eqref{eq:thm41_proof4b} and using the definition of $\delta_{+}$ in \eqref{eq:xi_func} we obtain
\begin{equation*}
g(\bar{y}^{+};\beta_1^{+}) \geq f(\bar{x}^{+};\beta_2^{+}) - \delta_{+},
\end{equation*}
where
$\delta_{+} := (1-\tau)\delta + 2\beta_1(1-\tau)D_{\sigma}\varepsilon_{[\sigma]} +
0.5\sum_{i=1}^ML_i^{\psi}(\beta_2^{+})\varepsilon_i^2 = (1-\tau)\delta + \xi(\tau, \beta_1, \beta_2, \varepsilon)$.
This is indeed \eqref{eq:excessive_gap_inexact} with the inexactness $\delta_{+}$.
\Eproof
%+ End of the proof.

%+ ******************************************************************
%+ 6.3. Proof of Lemma 3.1.
%+ ******************************************************************
\vskip0.2cm
\noindent\textit{A.3. The proof of Lemma \ref{le:init_point}. }
For simplicity of notation,  we denote by $x^{*} := x^{*}(0^m; \beta_1)$ and $\tilde{x}^{*} := \tilde{x}^{*}(0^m;\beta_1)$, $h(\cdot;y, \beta_1) :=
\sum_{i=1}^Mh_i(\cdot; y, \beta_1)$, where $h_i$ is defined in Definition \ref{de:inexactness}. By using the inexactness in inequality \eqref{eq:inexactness_xi}
and $y^c = 0^m$, we have $h(\tilde{x}^{*}; y, \beta_1) \leq h(x^{*}; y, \beta_1) + \frac{1}{2}\beta_1\varepsilon^2_{[\sigma]}$ which is rewritten as
\begin{eqnarray}\label{eq:lm31_proof1}
\phi(\tilde{x}^{*}) + \beta_1p_X(\tilde{x}^{*}) &&\leq \phi(x^{*}) +  \beta_1p_X(x^{*}) \overset{\tiny{\eqref{eq:smoothed_gi}}}{=} g(0^m; \beta_1) +
\frac{\beta_1}{2}\varepsilon_{[\sigma]}^2.
\end{eqnarray}
Since $g(\cdot;\beta_1)$ is concave, by using the estimate \eqref{eq:g_Lipschitz_bound} and $\nabla_yg(0^m;\beta_1) = Ax^{*} - b$ we have
\begin{align}\label{eq:lm31_proof2}
g(\bar{y}^0;\beta_1) &\geq g(0^m;\beta_1) + \nabla_yg(0^m;\beta_1)^T\bar{y}^0 - \frac{L^g(\beta_1)}{2}\norm{\bar{y}^0}^2 \nonumber\\
&= g(0^m;\beta_1) + (Ax^{*} - b)^T\bar{y}^0 - \frac{L^g(\beta_1)}{2}\norm{\bar{y}^0}^2 \nonumber\\
&\overset{\tiny\eqref{eq:lm31_proof1}}{\geq}
\phi(\tilde{x}^{*}) + \beta_1p_X(\tilde{x}^{*}) + (Ax^{*} - b)^T\bar{y}^0 - \frac{L^g(\beta_1)}{2}\norm{\bar{y}^0}^2 -
\frac{\beta_1}{2}\varepsilon_{[\sigma]}^2\\
&= \phi(\tilde{x}^{*}) + (A\tilde{x}^{*} - b)^T\bar{y}^0 - \frac{L^g(\beta_1)}{2}\norm{\bar{y}^0}^2 + (\bar{y}^0)^TA(x^{*} -\tilde{x}^{*}) +
\beta_1p_X(\tilde{x}^{*}) - \frac{\beta_1}{2}\varepsilon_{[\sigma]}^2 := T_0 \nonumber.
\end{align}
Since $\norm{x^{*} - \tilde{x}^{*}} \leq \varepsilon_{[1]}$, $p_X(\tilde{x}^{*}) \geq p_X^{*} > 0$ and $\bar{y}^0$ is the solution of \eqref{eq:psi_fun},
we estimate the last term $T_0$ of \eqref{eq:lm31_proof2} as
\begin{eqnarray}\label{eq:lm31_proof4}
T_0 ~&&~ \geq \phi(\tilde{x}^{*}) + \max_{y\in\mathbb{R}^m}\left\{(A\tilde{x}^{*} - b)^Ty - \frac{L^g(\beta_1)}{2}\norm{y}^2\right\} -
\norm{A^T\bar{y}^0}\norm{x^{*} - \tilde{x}^{*}} - \frac{\beta_1}{2}\varepsilon_{[\sigma]}^2 \nonumber\\
~&&~ \overset{\tiny\eqref{eq:init_point_cond}+\eqref{eq:accuracy_constants}}{\geq}
\phi(\tilde{x}^{*}) + \max_{y\in\mathbb{R}^m}\left\{(A\tilde{x}^{*} - b)^Ty - \frac{\beta_2}{2}\norm{y}^2\right\}  -\norm{A^T\bar{y}^0}\varepsilon_{[1]} -
\frac{\beta_1}{2}\varepsilon_{[\sigma]}^2 \\
~&&~ \overset{\tiny\eqref{eq:psi_approx}}{\geq}
f(\bar{x}^0; \beta_2) - \left[\norm{A^T\bar{y}^0}\varepsilon_{[1]} + \frac{\beta_1}{2}\varepsilon_{[\sigma]}^2\right].\nonumber
\end{eqnarray}
Now, we see that $p_X^{*} + \frac{\sigma_i}{2}\norm{\bar{x}_i^0 - x^c_i}^2 \leq p_{X_i}(\bar{x}_i^0) \leq \displaystyle\sup_{x_i\in X_i}p_{X_i}(x_i) = D_{X_i}$.
Thus, $\norm{\bar{x}_i^0 - x^c_i}^2\leq \frac{2}{\sigma_i}(D_{X_i}-p_{X_i}^{*}) \leq \frac{2D_{X_i}}{\sigma_i}$ for all $i=1,\dots, M$. By using the definition
of $D_{\sigma}$ in \eqref{eq:accuracy_constants}, the last inequalities imply
\begin{equation}\label{eq:lm31_proof4b}
\norm{\bar{x}^0 - x^c} \leq D_{\sigma}.
\end{equation}
Finally, we note that $A^T\bar{y}^0 = \frac{1}{L^g(\beta_1)}A^T(A\bar{x}^0 - b)$ due to \eqref{eq:start_point}. This relation leads to
\begin{eqnarray}\label{eq:lm31_proof5}
\norm{A^T\bar{y}^0} && = L^g(\beta_1)^{-1}\norm{A^T(A\bar{x}^0-b)} = L^g(\beta_1)^{-1}\norm{A^T(A(\bar{x}^0-x^c) + Ax^c - b)}\nonumber\\
&&\leq L^g(\beta_1)^{-1}\left[ \norm{A^TA}\norm{\bar{x}^0 - x^c} + \norm{A^T(Ax^c - b)}\right] \overset{\tiny\eqref{eq:lm31_proof4b}}{\leq}
L_{\mathrm{A}}^{-1}\beta_1\left[\norm{A}^2D_{\sigma} + \norm{A^T(Ax^c -
b)}\right] \\
&& \overset{\tiny\eqref{eq:accuracy_constants}}{=} L_{\mathrm{A}}^{-1}\beta_1C_d.\nonumber
\end{eqnarray}
By substituting \eqref{eq:lm31_proof4b} and \eqref{eq:lm31_proof5} into \eqref{eq:lm31_proof4} and then using the definition of
$\delta_0$ we obtain the conclusion of the lemma.
\Eproof
%+ End of the proof.

%+ ******************************************************************
%+ 6.4. The proof of Lemma 6.
%+ ******************************************************************
\vskip0.2cm
\noindent\textit{A.4. The proof of Lemma \ref{le:choice_of_tau}. }
%+ Proof of Lemma 6.
Let us consider the function $\xi(t;\alpha) := \frac{2}{\sqrt{t^3/(t-2\alpha)+1} + 1}$, where $\alpha \in [0, 1]$ and $t\geq 2$. After few simple
calculations, we can estimate $t + \alpha \leq \sqrt{t^3/(t-2\alpha) + 1} \leq t + 1$ for all $t > 2\max\{1, \alpha/(1-\alpha)\}$. These estimates lead to
\begin{equation*}
2(t+2)^{-1} \leq \xi(t;\alpha) \leq 2(t+1+\alpha)^{-1}~, \forall t > 2\max\{1, \alpha/(1-\alpha)\}.
\end{equation*}
From the update rule \eqref{eq:step_size_update} we can show that the sequence $\{\tau_k\}_{k\geq 0}$ satisfies $\tau_{k+1} := \xi(2/\tau_k;\alpha_k)$. If we
define $t_k := \frac{2}{\tau_k}$, then $\frac{2}{t_{k+1}} = \xi(\tau_k;\alpha_k)$. Therefore, one can estimate $t_k + 1 + \alpha_k \leq t_{k+1} \leq t_k + 2$
for $t_k > 2\max\{1, \alpha_k/(1-\alpha_k)\}$. Note that $\alpha_k\geq\alpha^{*}$ by Assumption \aref{as:A2} and by induction  we can show that $t_0 +
(1+\alpha^{*})k \leq t_k \leq t_0 + 2k$ for $k\geq 0$ and $t_0 > 2\max\{1, \alpha^{*}/(1-\alpha^{*})\}$. However, since $t_k = \frac{2}{\tau_k}$, this leads to:
\begin{equation*}
\frac{1}{k + 1/\tau_0} = \frac{1}{k+t_0/2} \leq \tau_k \leq \frac{1}{0.5(1+\alpha^{*})k + t_0/2} = \frac{1}{0.5(1+\alpha^{*})k + 1/\tau_0},
\end{equation*}
which is indeed \eqref{eq:tau_bounds}.
Here, $0 < \tau_0 = 2/t_0$ and $\tau_0 < [\max\{1, \alpha^{*}/(1-\alpha^{*})\}]^{-1}$.

In order to prove \eqref{eq:beta_bounds}, we note that $(1-\alpha_k\tau_k)(1-\tau_{k+1}) = \frac{\tau_{k+1}^2}{\tau_k^2}$. By induction, we have
$\prod_{i=0}^k(1-\alpha_k\tau_{k-1})\prod_{i=0}^k(1-\tau_k) = \frac{(1-\tau_0)\tau_k^2}{\tau_0^2}$. By combining this relation and the update rule
\eqref{eq:beta_update}, we  deduce that $\beta_1^k\beta_2^{k+1} = \beta_1^0\beta_2^0\frac{(1-\tau_0)\tau_k^2}{\tau_0^2}$, which is the third statement of
\eqref{eq:beta_bounds}.

Next, we prove the  bound on $\beta_1^k$. Since $\beta_1^{k+1} = \beta_1^0\prod_{i=0}^k(1-\alpha_i\tau_i)$, we have $\beta_1^0\prod_{i=0}^k(1-\tau_i) \leq
\beta_1^{k+1} \leq \beta_1^0\prod_{i=0}^k(1-\alpha^{*}\tau_i)$. By using the following elementary inequalities $-t-t^2 \leq \ln(1-t)\leq -t$ for all $t\in [0,
1/2]$, we obtain $\beta_1^0e^{-S_1 - S_2} \leq \beta_1^{k+1} \leq \beta_1^0e^{-\alpha^{*} S_1}$, where $S_1 := \sum_{i=0}^k\tau_i$ and $S_2 :=
\sum_{i=0}^k\tau_i^2$. From \eqref{eq:tau_bounds}, on the one hand, we have
\begin{equation*}
\sum_{i=0}^k\frac{1}{I+1/\tau_0} \leq S_1 \leq
\sum_{i=0}^k\frac{1}{0.5(1+\alpha^{*})i + 1/\tau_0},
\end{equation*}
which leads to $\ln(k+1/\tau_0)+\ln\tau_0 \leq S_1 \leq \frac{1}{0.5(1+\alpha^{*})}\ln(k+1/\tau_0) +
\gamma_0$ for some constant $\gamma_0$. On the other hand, we have $S_2$ converging to some constant $\gamma_2 > 0$. Combining all estimates together we
eventually get $\frac{\gamma}{(\tau_0k+1)^{2/(1+\alpha^{*})}} \leq \beta_1^{k+1} \leq \frac{\beta_1^0}{(\tau_0k+1)^{\alpha^{*}}}$ for some positive constant
$\gamma$.

Finally, we estimate the bound on $\beta_2^k$. Indeed, it follows from \eqref{eq:tau_bounds} that $\beta_2^{k+1} = \beta_2^0\prod_{i=0}^k(1-\tau_k) \leq
\beta_2^0\prod_{i=0}^k(1-\frac{1}{k+1/\tau_0}) = \beta_2^0\frac{1/\tau_0-1}{k+1/\tau_0} = \frac{\beta_2^0(1-\tau_0)}{\tau_0k + 1}$.
\Eproof
%+ End of the lemma.

%%%%%%%%%%%%%%%%%%%%%%%%%%%%%%%%%%%%%%%%%%%%%%%%%%%%%%%%%%%%%%%%%%%%%%%%%%%%%%%%%%%%%%%%%%%%%
%+ References.
%%%%%%%%%%%%%%%%%%%%%%%%%%%%%%%%%%%%%%%%%%%%%%%%%%%%%%%%%%%%%%%%%%%%%%%%%%%%%%%%%%%%%%%%%%%%%
\bibliographystyle{plain}

% \bibliography{tran_bibtex_new}

\end{document}